\documentclass[11pt,letterpaper,twoside,reqno,nosumlimits]{amsart}

\allowdisplaybreaks

\usepackage[usenames,dvipsnames]{xcolor}
\usepackage{fancyhdr}
\usepackage{amsmath,amsfonts,amsbsy,amsgen,amscd,amssymb,amsthm}
\usepackage{url}

\usepackage{mathtools}

\usepackage[font=small,margin=0.25in,labelfont={bf},labelsep={colon}]{caption}

\usepackage{subcaption}

\usepackage{tikz}
\usepackage{microtype}
\usepackage{enumitem}

\definecolor{dark-gray}{gray}{0.3}
\definecolor{dkgray}{rgb}{.4,.4,.4}
\definecolor{dkblue}{rgb}{0,0,.5}
\definecolor{medblue}{rgb}{0,0,.75}
\definecolor{rust}{rgb}{0.5,0.1,0.1}

\usepackage[colorlinks=true]{hyperref}

\hypersetup{urlcolor=rust}
\hypersetup{citecolor=dkblue}
\hypersetup{linkcolor=dkblue}

\usepackage{setspace}

\usepackage{graphicx}
\usepackage{booktabs,longtable,tabu} 
\usepackage{multirow} 

\usepackage{float}

\usepackage[full]{textcomp}

\usepackage[scaled=.98,sups]{XCharter}
\usepackage[scaled=1.04,varqu,varl]{inconsolata}
\usepackage[type1]{cabin}
\usepackage[charter,vvarbb,scaled=1.07]{newtxmath}
\usepackage[cal=boondoxo]{mathalfa}
\usepackage[T1]{fontenc}

\usepackage{bm} 

\graphicspath{{figures/}}

\newtheorem{theorem}{Theorem}[section]
\newtheorem{lemma}[theorem]{Lemma}

\newtheorem{proposition}[theorem]{Proposition}

\newtheorem{corollary}[theorem]{Corollary}

\theoremstyle{definition}

\numberwithin{equation}{section} 
\numberwithin{figure}{section}
\numberwithin{table}{section}

\floatstyle{plaintop}
\newfloat{recipe}{thp}{lor}
\floatname{recipe}{Recipe}
\numberwithin{recipe}{section}

\providecommand{\mathbold}[1]{\bm{#1}}  
                                
\newcommand{\suml}{\sum\limits}

\providecommand{\mathbbm}{\mathbb} 

\newcommand{\R}{\mathbbm{R}}

\newcommand{\Z}{\mathbbm{Z}}

\newcommand{\sgn}[1]{\operatorname{sgn}{#1}}

\newcommand{\diff}[1]{\mathrm{d}{#1}}
\newcommand{\idiff}[1]{\, \diff{#1}}

\newcommand{\mtx}[1]{\mathbold{#1}}

\newcommand{\triplenorm}[1]{{\left\vert\kern-0.25ex\left\vert\kern-0.25ex\left\vert #1
    \right\vert\kern-0.25ex\right\vert\kern-0.25ex\right\vert}}

\usepackage[numbers]{natbib}

\newcommand{\om}{\omega}
\newcommand{\lam}{\lambda}
\renewcommand{\d}{\delta}

\newcommand{\la}{\langle}
\newcommand{\ra}{\rangle}
\newcommand{\pa}{\partial}

\mathtoolsset{showonlyrefs}

\begin{document}

\title[Self-similar solutions of De Gregorio model]{On self-similar finite-time blowups of\\ the De Gregorio model on the real line}
\author[D. Huang, J. Tong, and D. Wei]{De Huang$^1$, Jiajun Tong$^2$, and Dongyi Wei$^3$}
\thanks{$^1$School of Mathematical Sciences, Peking University. E-mail: dhuang@math.pku.edu.cn}
\thanks{$^2$Beijing International Center for Mathematical Research, Peking University. E-mail: tongj@bicmr.pku.edu.cn}
\thanks{$^3$School of Mathematical Sciences, Peking University. E-mail: jnwdyi@pku.edu.cn}

\begin{abstract}
We show that the De Gregorio model on the real line admits infinitely many compactly supported, self-similar solutions that are distinct under rescaling and will blow up in finite time. These self-similar solutions fall into two classes: the basic class and the general class. The basic class consists of countably infinite solutions that are eigenfunctions of a self-adjoint compact operator. In particular, the leading eigenfunction coincides with the finite-time singularity solution of the De Gregorio model recently obtained by numerical approaches. The general class consists of more complicated solutions that can be obtained by solving nonlinear eigenvalue problems associated with the same compact operator. 
\end{abstract}

\maketitle

\section{Introduction}
We consider the $1$D De Gregorio equation
\begin{equation}\label{eqt:DG}
\om_t + u\om_x = u_x \om, \quad u_x = \mtx{H}(\om),\quad u(0)=0,
\end{equation}
for $x\in \R$, where $\mtx{H}(\cdot)$ denotes the Hilbert transform on the real line. The normalization condition $u(0)=0$ is not essential; we impose it throughout the paper to remove the degree of freedom due to translation. \eqref{eqt:DG} is a simplified model for the vorticity formulation of the $3$D incompressible Euler equations, proposed by De Gregorio \cite{de1990one,de1996partial} to study the competition between advection and vortex stretching. In particular, $\om$ models the vorticity, and the nonlinear terms $u\om_x$ and $u_x \om$ model the advection term and the vortex stretching term, respectively. The $3$D Biot-Savart law that recovers the velocity from the vorticity is modeled by $u_x= \mtx{H}(\om)$, which has the same scaling as the original Biot-Savart law.

The fundamental question on the global regularity of the $3$D Euler equations with smooth initial data of finite energy remains one of the most challenging open problems in fluid dynamics and nonlinear partial differential equations. It is widely believed that the vortex stretching effect has the potential to induce an infinite growth of the vorticity in finite time. The advection, on the contrary, has been found to have a smoothing effect that may weaken the local growth of the solution and destroy the potential singularity formation (e.g., see \cite{okamoto2005role,hou2006dynamic,hou2008dynamic,lei2009stabilizing}). As an exploratory version of the bigger question, it is natural to ask whether the De Gregorio model, with the presence of both effects, can develop a finite-time singularity from smooth initial data with finite energy. This $1$D problem was recently settled by Chen, Hou and Huang \cite{chen2021finite} using a computer-aided proof. They proved the existence of a self-similar solution $\om_s$ that expands spatially and blows up in finite time, and they show that any solution that is initially close to $\om_s$ in some weighted $H^1$-norm shall develop an asymptotically self-similar singularity with the same scaling. 

By the scaling property of equation \eqref{eqt:DG}, a self-similar solution exhibiting blowup at a finite time $T$ shall take the form
\begin{equation}\label{eqt:self-similar_solution}
\om_s(x,t) = (T-t)^{c_\om}\cdot\Omega\left(\frac{x}{(T-t)^{c_l}}\right),
\end{equation}
where $\Omega$ is referred as the self-similar profile, and $c_\om,c_l$ are the scaling factors. Plugging this formula into \eqref{eqt:DG} and taking $t\rightarrow T$ yields that the only possible non-zero value for $c_\om$ is $-1$. The value of $c_l$ determines the spatial feature of $\om_s$: The case $c_l>0$ corresponds to a focusing blowup at $x=0$, while a negative $c_l$ corresponds to an expanding blowup. The self-similar solution constructed in \cite{chen2021finite} is an expanding one with $c_l = c_\om = -1$ and $\Omega=\Omega_*\in H^1(\R)$. As a crucial intermediate step in their construction, a highly accurate approximate self-similar profile $\widetilde{\Omega}_*$ that is odd in $x$, compactly supported, and non-positive on $\R_+$ is obtained by numerically solving the dynamic rescaling equation of \eqref{eqt:DG}. Based on the properties of the approximate self-similar profile $\widetilde{\Omega}_*$ and a nonlinear stability argument, it is shown that the exact self-similar profile $\Omega_*$ is also an odd function that is compactly supported and non-positive on $\R_+$. Though $\Omega_*$ is not smooth on $\R$, its stability in some weighted $H^1$-norm guarantees that any smooth solution that is sufficiently close to $\Omega_*$ will also blow up asymptotically self-similarly in finite time, hence proving the finite-time singularity of \eqref{eqt:DG} with smooth initial data of finite energy. In \cite{lushnikov2021collapse}, Lushnikov, Silantyev and Siegel further verified the existence of an expanding self-similar blowup of \eqref{eqt:DG} with $c_l=c_\om=-1$ using highly accurate numerical computations. In this paper, we supplement the previous results by proving the existence of the self-similar profile obtained in \cite{chen2021finite} through a different approach. Moreover, we will show that it is not the only self-similar profile with the same blowup scaling. 

\begin{theorem}\label{thm:main_theorem_informal}
The De Gregorio model \eqref{eqt:DG} on the real line admits infinitely many self-similar solutions of the form \eqref{eqt:self-similar_solution}, distinct under rescaling, with $c_l = c_\om = -1$ and odd self-similar profiles $\Omega\in H^1(\R)$ supported on $[-1,1]$, among which there is a unique self-similar profile that is strictly negative on $(0,1)$.
\end{theorem}

We will give detailed characterizations of these self-similar solutions, based on the crucial observation that they all correspond to some linear or nonlinear eigenvalue problems of a self-adjoint compact operator $\mtx{M}\in \mathcal{L}(\mathbb{V})$. Here $\mathbb{V}$ is the space of all odd $H^1$ functions that are supported on $[-1,1]$, and $\mtx{M}$ is defined as $\mtx{M}(f)(x) = \chi_{[-1,1]}\big((-\Delta)^{-1/2}f(x) - (-\Delta)^{-1/2}f(1)\cdot x\big)$ for any $f\in \mathbb{V}$. More precisely, we can find self-similar profiles in two classes: 

\begin{enumerate}
\item The basic class (of countably infinite solutions): $\Omega$ satisfies the eigenvalue problem 
\[\mtx{M}(\Omega) = \lambda\Omega,\quad x\in[-1,1],\]
for some positive eigenvalue $\lambda>0$.
\label{basic_class}
\item The general class (of uncountably infinite solutions): $\Omega$ satisfies the piecewise eigenvalue problem 
\[\mtx{M}(\Omega) = \mu_i\Omega, \quad x\in [-x_i,-x_{i-1}]\cup[x_{i-1},x_i],\quad i=1,\dots,n,\]
for some integer $n\geq 2$, where $0=x_0<x_1<\dots<x_{n-1}<x_n=1$, and $\{\mu_i\}_{i=1}^n$ is a sequence of positive numbers such that $\mu_i\neq \mu_{i-1}$.
\label{general_class}
\end{enumerate}
We shall prove that the only self-similar profile $\Omega_*$ that is strictly negative on $(0,1)$ coincides with the one obtained in \cite{chen2021finite,lushnikov2021collapse} (up to rescaling), which corresponds to the largest eigenvalue in the basic class.

When finalizing this paper, we learned that Jia and Sverak have obtained similar results in an independent ongoing work. They have also found infinite solutions, including the special one $\Omega_*$, in the basic class by reducing the nonlinear problem to a Steklov eigenvalue problem, and they conjectured the existence of solutions in the general class. See \cite{VideoLecture} for a recent video lecture by Sverak on this topic. \\

The De Gregorio model belongs to the widely-studied one-parameter family of the generalized Constantin--Lax--Majda equation on the real line
\begin{equation}\label{eqt:gCLM}
\om_t + au\om_x = u_x \om, \quad u_x = \mtx{H}(\om),\quad u(0)=0.
\end{equation}
This line of research is also motivated by the study of finite-time singularities in the $3$D incompressible Euler equations. The equation with $a=0$ was first introduced by Constantin, Lax and Majda \cite{constantin1985simple}, followed by the generalization of De Gregorio \cite{de1990one} to include an advection term $u\om_x$. Later, Okamoto, Sakajo and Wensch \cite{okamoto2008generalization} introduced the real parameter $a$ to modify the effect of advection in the competition against vortex stretching. We refer to \cite{elgindi2020effects,lushnikov2021collapse} for more background information on this subject. 

In addition to the De Gregorio case $a=1$, the singularity formation of \eqref{eqt:gCLM} for a wide range of $a$ has been studied extensively. In the regime $a<0$, the advection term works in favor of producing a singularity. Indeed, Castro and Cordoba \cite{castro2010infinite} proved finite-time blowup for all $a<0$ based on a Lyapunov functional argument. In particular, finite-time singularity for the special case $a=-1$ was established earlier by Cordoba, Cordoba and Fontelos \cite{cordoba2005formation}. In the case $a=0$ of no advection, finite-time self-similar singularity of the form \eqref{eqt:self-similar_solution} with $c_l = -c_\om = 1$ was established by Constantin, Lax and Majda \cite{constantin1985simple} via the construction of a closed-form exact solution to \eqref{eqt:gCLM}. Based on this exact solution, Elgindi and Jeong \cite{elgindi2020effects} proved the existence of finite-time self-similar singularities from smooth initial data for $|a|$ small enough using a continuation argument. As mentioned above, the finite-time self-similar singularity for $a=1$ was recently proved in \cite{chen2021finite}. Later, Lushnikov, Silantyev, Siegel \cite{lushnikov2021collapse} and Chen \cite{chen2020singularity} independently found an exact self-similar solution for $a=1/2$ with $c_l = -c_\om/3 = 1/3$. Chen also proved finite-time self-similar singularities from smooth initial data for $a$ close to $1/2$ using the method developed in \cite{chen2021finite}. Other than the settled cases of $a=0,1/2,1$ and $a$ close to these three values, it remains an open problem whether \eqref{eqt:gCLM} admits finite-time self-similar singularities of the form \eqref{eqt:self-similar_solution} with $\Omega\in H^1(\R)$ for all $a\in[0,1]$. Nevertheless, the numerical studies by Lushnikov, Silantyev and Siegel \cite{lushnikov2021collapse} suggested the existence of a family of self-similar solutions $\om_s^{(a)}$, with $\Omega^{(a)}$ and $c_l^{(a)}$ continuously depending on $a$. In particular they discovered a critical value $a_c \approx 0.689$ such that $c_l^{(a)}<0$ for $a>a_c$ while $c_l^{(a)}>0$ for $a<a_c$. That is, $a_c$ is the transition threshold that separates focusing singularities from expanding ones. Finally, we remark that self-similar solutions with H\"older continuous profiles $\Omega$ have been constructed in \cite{elgindi2020effects} for all values of $a$, based on which finite-time self-similar blowup from H\"older continuous initial data with finite energy was proved in \cite{chen2021finite}. 

Many parallel studies have also considered the generalized Constantin--Lax--Majda equation \eqref{eqt:gCLM} on the circle $\mathbb{S}^1$. Although exact self-similar solutions of the form \eqref{eqt:self-similar_solution} cannot exist on the circle, focusing finite-time self-similar singularity for $a<a_c$ can still develop in an asymptotic way with the same $\Omega,c_l,c_\om$ as on the real line \cite{chen2021finite,lushnikov2021collapse,chen2020singularity}. On the contrary, expanding self-similar singularities with $c_l<0$ are clearly incompatible with the periodic setting. In fact, global-in-time well-posedness of the De Gregorio model on $\mathbb{S}^1$ was earlier conjectured and supported by numerical simulations \cite{de1996partial,lushnikov2021collapse,okamoto2008generalization} and was recently proved by Chen \cite{chen2021regularity}. Besides, in a recent work of Jia, Stewart and Sverak \cite{jia2019gregorio}, local nonlinear stability was established for the equilibria $A\sin(x-x_0)$ of \eqref{eqt:DG} based on spectral theories and complex analysis. An alternative proof was later provided by Lei, Liu and Ren \cite{lei2020constantin} using a direct energy method. As an interesting perspective, we will discuss the similarities between the equilibrium $\sin(x)$ on $\mathbb{S}^1$ and the self-similar profile $\Omega_*$ on $\R$, and we will remark on the connections between their stabilities. As a generalization of the nonlinear stability at $a=1$, Chen \cite{chen2021slightly} established finite-time self-similar singularity of \eqref{eqt:gCLM} on the circle with $c_l=0$ (neither focusing nor expanding) for $a$ strictly smaller than and sufficiently close to $1$, whose self-similar profile converges to $-\sin(x)$ as $a$ approaches $1$.

It is worth mentioning that singularity formation and global well-posedness for the generalized Constantin--Lax--Majda equation with dissipation have also been extensively studied in the literature \cite{schochet1986explicit,kiselev2010regularity,li2008blow,silvestre2016transport,cordoba2005formation,dong2008well,wunsch2011generalized}.\\

The remaining of this paper is organized as follows. In Section \ref{sec:preparation}, we prove a necessary condition for the self-similar profiles of the De Gregorio model with $c_l=c_\om$, and we show that they must be compactly supported. Our main result, a more formal version of Theorem \ref{thm:main_theorem_informal}, will be stated in Section \ref{sec:main_result}. Section \ref{sec:basic_class} and Section \ref{sec:general_class} contribute to constructing self-similar profiles in the basic class \eqref{basic_class} and in the general class \eqref{general_class}, respectively. In Section \ref{sec:nonlinear_stability}, we discuss the connections between the nonlinear stability results established in \cite{chen2021finite} and \cite{jia2019gregorio}.  

\subsection*{Acknowledgement} The authors are supported by the National Key R\&D Program of China under the grant 2021YFA1001500.

\section{The equation for the self-similar profiles}\label{sec:preparation}
In this section, we derive an equation for the self-similar profiles, and we explain why we seek for solutions that are compactly supported. We then relate the self-similar profiles with $c_l=c_\om$ to an eigenvalue problem. 

Substituting the self-similar ansatz \eqref{eqt:self-similar_solution} into the equation \eqref{eqt:gCLM} yields 
\[-c_\om(T-t)^{c_\om-1}\Omega + c_l(T-t)^{c_\om-1}X\Omega_X + (T-t)^{2c_\om}U\Omega_X = (T-t)^{2c_\om}U_X\Omega,\]
where $X = x/(T-t)^{c_l}$ and $U_X(X) = \mtx{H}(\Omega)(X)$. Provided that $c_\om\neq 0$, balancing the equation above as $t\rightarrow T$ yields $c_\om = -1$ and an equation for the self-similar profile:
\[(c_lX+U)\Omega_X  = (c_\om+U_X)\Omega,\quad U_X = \mtx{H}(\Omega),\quad U(0) = 0.\]

For notational simplicity, we will still use $\om,u,x$ for $\Omega, U, X$, respectively. Our goal is to find suitable solutions $(\om, c_l, c_{\om})$ to the equation 
\begin{equation}\label{eqt:main_equation}
(c_lx+u)\om_x  = (c_\om+u_x)\om,\quad u_x = \mtx{H}(\om),\quad u(0)=0.
\end{equation}
The expressions of $u$ and $u_x$ in terms of $\om$ are, respectively, 
\begin{align*}
u(x) &= -(-\Delta)^{-1/2}\om(x) =  \frac{1}{\pi}\int_{\R}\om(y)\ln|x-y|\idiff y,\\
u_x(x) &= \mtx{H}(\om)(x) = \frac{1}{\pi}P.V.\int_{\R}\frac{\om(y)}{x-y}\idiff y.
\end{align*}

Note that if $(\om(x), c_l, c_\om)$ is a solution to \eqref{eqt:main_equation}, then
\begin{equation}\label{eqt:scaling}
(\om_{\alpha,\beta}(x),\ c_{l,\alpha},\ c_{\om,\alpha})= (\alpha\omega(\beta x),\ \alpha c_l,\ \alpha c_\om)
\end{equation}
is also a solution for any $\alpha, \beta \in \R, \beta\neq 0$. Therefore, we will release ourselves from the restriction $c_\om=-1$. In fact, it is the ratio $c_l/c_\om$ that matters. Our question is, for what values of $c_l/c_\om$ does equation \eqref{eqt:main_equation} admit a suitable solution $\om(x)$ that is sufficiently regular and integrable? More precisely, in order to prove Theorem \ref{thm:main_theorem_informal}, we look for solutions that satisfy the following conditions:
\begin{itemize}
\item Odd symmetry: $\om(x)$ is an odd function of $x$, i.e. $\om(-x) = -\om(x)$.
\item Regularity: $\om\in H^1(\R)$.
\item Non-degeneracy: $\om_x(0) \neq 0$.
\end{itemize}

The odd symmetry, which is preserved by \eqref{eqt:DG}, is a common feature of all finite-time self-similar singularities of the generalized Constantin--Lax--Majda equation that have been found so far in the literature. Moreover, it is proved in \cite{lei2020constantin} that the De Gregorio model \eqref{eqt:DG} is globally well-posed for initial data that does not change sign on $\R$ (under some mild regularity assumption). Therefore, we only focus on odd solutions.

The condition $\om\in H^1(\R)$ serves two purposes. The first one is to avoid solutions with relatively lower regularity. Elgindi and Jeong \cite{elgindi2020effects} have proved the existence of self-similar solutions to \eqref{eqt:gCLM} with $C^\alpha$ profiles for some small $\alpha=\alpha(a)\in(0,1)$ for all values of $a$. Our goal is to prove the existence of self-similar profiles with higher regularity. In fact, we will show that any solution $\om$ to \eqref{eqt:main_equation} in the basic class \eqref{basic_class} is infinitely smooth in the interior of $[-1,1]$. The second purpose of this condition is to exclude solutions that do not decay sufficiently fast at infinity. As we will see, the fast decay of a profile solution $\om$ leads to some \textit{a priori} estimates on the asymptotic behaviors of $u$ and $u_x$ that in turn imply $\om$ must be compactly supported.

In view of the rescaling property \eqref{eqt:scaling}, the non-degeneracy condition $\om_x(0)\neq 0$ is to make sure that $\om$ is non-trivial. A direct consequence of this condition is that 
\[c_l = c_\om.\]
To see this, we simply divide \eqref{eqt:main_equation} by $x$ and take the limit $x\rightarrow 0$. In the rest of this paper, we will always assume that $c_l = c_\om$. Conversely, we will show that any solutions in either the basic class \eqref{basic_class} or the general class \eqref{general_class} must have non-zero derivative at $x=0$. However, we remark that our result does not exclude the possibility of a non-trivial odd solution $\om\in H^1(\R)$ to \eqref{eqt:main_equation} that satisfies $\om_x(0)=0$ (possibly with a different value of $c_l/c_\om$).

\subsection{A necessary condition} 
Let us start with a necessary condition for any solution $\om\in H^1(\R)$ to \eqref{eqt:main_equation} with $c_l = c_\om$.

\begin{proposition}\label{prop:necessary_condition}
Suppose that $\om\in H^1(\R)$ is a solution to \eqref{eqt:main_equation} with $c_l = c_\om$. Then for any interval $[x_1,x_2]$ such that $|\om(x)|>0$ on $[x_1,x_2]$, there is some constant $C_1\neq \pm\infty$ such that 
\begin{equation}\label{eqt:necessary_condition_1}
C_1\,\om(x) = u(x) + c_\om x ,\quad x\in [x_1,x_2].
\end{equation}
Similarly, if $|u(x) + c_\om x|>0$ on $[x_1,x_2]$, there is some constant $C_2\neq \pm\infty$ such that 
\begin{equation}\label{eqt:necessary_condition_2}
\om(x) = C_2\big(u(x) + c_\om x\big) ,\quad x\in [x_1,x_2].
\end{equation}
\end{proposition}

\begin{proof}
Note that $\om\in H^1(\R)$ implies $u\in H^2(\R)$. We can rearrange \eqref{eqt:main_equation} with $c_l=c_\om$ to find that
\[\partial_x(u(x) + c_\om x) = \frac{\om_x(x)}{\om(x)}(u(x) + c_\om x),\quad x\in [x_1,x_2].\]
Solving this ODE for $u(x) + c_\om x$ yields \eqref{eqt:necessary_condition_1} with $C_1 = (u(x_1) + c_\om x_1)/\om(x_1)\neq \pm\infty$. Moreover, if $|u(x) + c_\om x|>0$ on $[x_1,x_2]$, then 
\[\om_x(x) = \frac{u_x(x) + c_\om}{u(x) + c_\om x}\,\om(x),\quad x\in [x_1,x_2],\]
which implies $\om\in C^1([x_1,x_2])$. Hence, we can solve this ODE for $\om$ to obtain \eqref{eqt:necessary_condition_2} with $C_2 = \om(x_1)/(u(x_1) + c_\om x_1)\neq \pm \infty$.
\end{proof}

\subsection{Compact support} 
We explain why an odd solution $\om\in H^1(\R)$ to the self-similar profile equation \eqref{eqt:main_equation} with $c_l=c_\om\neq 0$ must have a compact support. To do so, we first prove a decay property of the corresponding $u$ and $u_x$ in the far field.

\begin{lemma}\label{lem:decay_property}
Suppose that $\om(x)\in H^1(\R)$ is an odd function of $x$. Then $u(x)$ is odd in $x$, and $u_x(x)$ is even in $x$. Moreover, $u(x)/x\rightarrow 0$ and $u_x(x)\rightarrow0$ as $x\rightarrow \pm\infty$.
\end{lemma} 

\begin{proof}
By the odd symmetry of $\om$, we have
\[u(x) = \frac{1}{\pi}\int_{\R}\om(y) \ln|x-y| \idiff y  = \frac{1}{\pi}\int_0^\infty \om(y) \ln\left|\frac{x-y}{x+y}\right| \idiff y.\]
Therefore, for any $x>0$, 
\begin{align*}
\left|\frac{u(x)}{x}\right| &\lesssim  \left|\int_0^\infty\frac{\om(y)}{x}\ln\left|\frac{x-y}{x+y}\right|\idiff y\right|\\
&\leq \|\om\|_{L^2(\R)}\cdot\left(\int_0^\infty\frac{1}{x^2}\left(\ln\left|\frac{x-y}{x+y}\right|\right)^2\idiff y\right)^{1/2}\\
&= \|\om\|_{L^2(\R)} \cdot \frac{1}{\sqrt{x}}\cdot \left(\int_0^\infty\left(\ln\left|\frac{1-t}{1+t}\right|\right)^2\idiff t\right)^{1/2}.
\end{align*}
Note that the last integral is finite. Hence, $u(x)/x\rightarrow 0$ as $x\rightarrow \infty$.

It is a bit trickier to show that $\lim_{x\rightarrow +\infty}u_x(x)=0$. For any $x>1$, we have 
\begin{align*}
u_x(x) &= \frac{1}{\pi}P.V.\int_{\R}\frac{\om(y)}{x-y}\idiff y = \frac{1}{\pi}P.V.\int_0^\infty\om(y)\left(\frac{1}{x-y}-\frac{1}{x+y}\right)\idiff y\\
&=\frac{1}{\pi}P.V.\int_{x-1}^{x+1}\frac{\om(y)}{x-y}\idiff y + \frac{1}{\pi}\int_0^{x-1}\frac{\om(y)}{x-y}\idiff y + \frac{1}{\pi}\int_{x+1}^\infty\frac{\om(y)}{x-y}\idiff y - \frac{1}{\pi}\int_0^\infty\frac{\om(y)}{x+y}\idiff y\\
&=: I_1 + I_2 + I_3 + I_4.
\end{align*}
For the first term $I_1$, we use integration by part to reach 
\begin{align*}
I_1 &= -\frac{1}{\pi}\om(y)\ln|x-y|\Big|_{x-1}^{x+1} + \frac{1}{\pi}\int_{x-1}^{x+1}\om_x(y)\ln\left|x-y\right|\idiff y\\
&\lesssim \|\om\|_{H^1([x-1,x+1])}\cdot \left(\int_{-1}^1\left(\ln|t|\right)^2\idiff t\right)^{1/2} \lesssim \|\om\|_{H^1([x-1,x+1])}.
\end{align*}
It follows from the assumption $\|\om\|_{H^1(\R)}<+\infty$ that $\lim_{x\rightarrow +\infty}I_1=0$. As for $I_2$, we further decompose it to reach
\begin{align*}
|I_2| &\lesssim \int_0^{x-\sqrt{x}}\left|\frac{\om(y)}{x-y}\right|\idiff y + \int_{x-\sqrt{x}}^{x-1}\left|\frac{\om(y)}{x-y}\right|\idiff y \\
&\leq  \|\om\|_{L^2(\R)}\cdot \left(\int_0^{x-\sqrt{x}}\frac{1}{(x-y)^2}\idiff y\right)^{1/2} + \|\om\|_{L^2([x-\sqrt{x},x-1])}\cdot \left(\int_{x-\sqrt{x}}^{x-1}\frac{1}{(x-y)^2}\idiff y\right)^{1/2}\\
&\lesssim  x^{-1/4}\|\om\|_{L^2(\R)} + \|\om\|_{L^2([x-\sqrt{x},x-1])}.
\end{align*}
Similarly, we can show that $|I_3| \lesssim \|\om\|_{L^2([x+1,\infty))}$. Since $\|\om\|_{L^2(\R)}<+\infty$, we have $|I_2|,|I_3| \rightarrow 0$ as $x\rightarrow+\infty$. Finally, we control $I_4$ as
\[|I_4|\lesssim \|\om\|_{L^2(\R)}\cdot \left(\int_0^\infty\frac{1}{(x+y)^2}\idiff y\right)^{1/2} \lesssim \frac{1}{\sqrt{x}}\|\om\|_{L^2(\R)}.\]
Combining the estimates above yields $\lim_{x\rightarrow +\infty}u_x(x)=0$.
\end{proof}

The sub-linear behavior of $u(x)$ provides more information on the far-field behavior of $\om(x)$. Suppose that $c_\om/c_l>0$. By Lemma \ref{lem:decay_property}, for any $0<\epsilon<c_\om/c_l$, there is some $x_0>0$ such that
\[\frac{c_\om + u_x(x)}{c_lx + u(x)} \geq \left(\frac{c_\om}{c_l}-\epsilon\right)\frac{1}{x},\quad x\geq x_0.\]
Now suppose that $\om\in H^1(\R)$ is not compactly supported. Then there must be some $x_2>x_1\geq x_0$ such that $\om(x)>0$ on $[x_1,x_2]$ (by choosing the sign of $\om$), and hence
\[\frac{\om_x(x)}{\om(x)} = \frac{c_\om + u_x(x)}{c_lx + u(x)} \geq \left(\frac{c_\om}{c_l}-\epsilon\right)\frac{1}{x},\quad x\in [x_1,x_2].\]
Solving this yields that 
\[\frac{\om(x_2)}{x_2^{c_\om/c_l-\epsilon}}\geq \frac{\om(x_1)}{x_1^{c_\om/c_l-\epsilon}}.\]
Extending this estimate towards infinity shows that
\[\om(x)\geq \frac{\om(x_1)}{x_1^{c_\om/c_l-\epsilon}}\, x^{c_\om/c_l-\epsilon},\quad x\geq x_1,\]
which contradicts with the assumption $\om \in H^1(\R)$ provided that $0<\epsilon<c_\om/c_l$. That is, we have proved the following, which apparently applies to the case $c_l=c_\om$.

\begin{proposition}\label{prop:compact_support}
Any solution $\om\in H^1(\R)$ to \eqref{eqt:main_equation} with $c_\om/c_l>0$ must be compactly supported. 
\end{proposition}

By the rescaling property \eqref{eqt:scaling}, we only need to look for compactly supported solutions $\om$ such that  
\begin{equation}\label{eqt:support_size}
\inf\,\{x>0: \om \equiv0\ \text{on}\ \R\backslash(-x,x)\} = 1.
\end{equation}
This normalization condition and Proposition \ref{prop:necessary_condition} together determine the value of $c_\om$ from $\om$.

\begin{proposition}\label{prop:c_omega}
Let $\om\in H^1(\R)$ be a solution to \eqref{eqt:main_equation} with $c_l=c_\om$. Suppose that $\om$ satisfies the condition \eqref{eqt:support_size}. Then, 
\begin{equation}\label{eqt:c_omega}
c_\om = - u(1) = -\frac{1}{\pi}\int_{\R}\om(y)\ln|1-y|\idiff y. 
\end{equation}
\end{proposition}

\begin{proof}
We only need to show that $u(1) + c_\om = 0$. Assume that this is not true. By the continuity of $u$, there is some $x_0<1$ such that $|u(x) + c_\om x|>0$ on $[x_0,1]$. Then from Proposition \ref{prop:necessary_condition} we know that there is some constant $C\neq \pm\infty$ such that 
\[\om(x) = C\big(u(x) + c_\om x\big),\quad x\in [x_0,1].\]
However, since $\om(1) = 0$, we must have $C=0$, and thus $\om\equiv 0$ on $[x_0,1]$. This contradicts with the normalization condition \eqref{eqt:support_size}.
\end{proof}

\subsection{Main result}\label{sec:main_result}
Provided the discussions above, particularly Propositions \ref{prop:necessary_condition}, \ref{prop:compact_support} and \ref{prop:c_omega}, we now restate Theorem \ref{thm:main_theorem_informal} more formally as follows.

\begin{theorem}\label{thm:main_theorem}
The equation \eqref{eqt:main_equation} of the self-similar profiles for the De Gregorio model admits infinitely many solutions $(\om,c_l,c_\om)$ (distinct with respect to the rescaling \eqref{eqt:scaling}) such that $\om\in H^1(\R)$ is odd in $x$ and compactly supported on $[-1,1]$, and $c_l=c_\om\neq 0$ is given by \eqref{eqt:c_omega}. Moreover, these solutions fall into the following two classes:
\begin{enumerate}
\item \label{basic_claiss_formal} \textbf{The basic class}. There are countably infinite solutions that satisfy the followings:
\begin{itemize}
\item $\om$, $u$, and $c_\om$ satisfy the relation 
\begin{equation}\label{eqt:eigen_form}
-\lambda\, \om(x) = u(x) + c_\om x,\quad x\in [-1,1],
\end{equation}
for some constant $\lambda>0$. 
\item $\om\in H^1_0([-1,1])\cap H^2([-1,1])$, and $\om$ is smooth in the interior of $(-1,1)$. 
\end{itemize}
There is a unique solution $\om_*$ (up to a multiplicative constant) in this class that is strictly positive on $(0,1)$; all the other solutions in this class are sign-changing on $(0,1)$.
\item \label{general_claiss_formal} \textbf{The general class}. For each integer $n\geq 2$, there are uncountably infinite solutions that satisfy the followings: 
\begin{itemize}
\item $\om$, $u$, and $c_\om$ satisfy the relation 
\[-\mu_i\, \om(x) = u(x) + c_\om x, \quad x\in [-x_i,-x_{i-1}]\cup[x_{i-1},x_i],\quad i=1,\dots,n,\]
for some points $0=x_0<x_1<\dots<x_{n-1}<x_n=1$ and some sequence $\{\mu_i\}_{i=1}^n$ of positive numbers such that $\mu_i\neq \mu_{i-1}$.
\item $\om\in H^1_0([-1,1])$. For each $i=1,\dots,n$, $\om\in H^2([x_{i-1},x_i])$, and $\om$ is smooth in the interior of $(x_{i-1},x_i)$. 
\item $\om$ is sign-changing on $(0,1)$.
\end{itemize}
\end{enumerate}
\end{theorem}

The remaining of this paper contributes to the proof of Theorem \ref{thm:main_theorem}. We will deal with the basic class in Section \ref{sec:basic_class} and the general class in Section \ref{sec:general_class}. The basic class solutions are obtained as the eigenfunctions of some compact operator, and the general class solutions are constructed as perturbations of these eigenfunctions. In particular, for each integer $n\geq2$ and each basic class solution $\om$ that has at least $2n+1$ zeros in $[-1,1]$, there is a neighborhood of $\om$ and a continuous family of $n$-tuples $(\mu_1,\dots,\mu_n)$, such that each tuple in it determines a unique general class solution in that neighborhood up to a multiplicative constant. We emphasize that, for each solution we find, the corresponding value of $c_l=c_\om$ is non-zero, so that it can be rescaled to become a solution of \eqref{eqt:main_equation} with $c_l=c_\om = -1$.

\section{The basic class}\label{sec:basic_class}
This section is devoted to proving the first part of Theorem \ref{thm:main_theorem}: The existence of countably infinite solutions in the basic class \eqref{basic_claiss_formal}, including a special solution that does not change sign on $(0,1)$. We do so by formulating the problem as an eigen-problem for a self-adjoint compact operator, and we will give detailed characterization of the eigenvalues and eigenfunctions.

\subsection{An eigenvalue problem}\label{sec:eigenproblem}
Consider the Hilbert space of odd $H^1$ functions supported on $[-1,1]$,
\[\mathbb{V} := \left\{f:\ f(x) = -f(-x),\ f\in H^1_0([-1,1]) \right\},\]
endowed with the standard $\dot{H}^1$ inner product $\la\cdot\,,\,\cdot\ra_{\dot{H}^1}$. Guided by the desired relation \eqref{eqt:eigen_form}, we define a linear map $\mtx{M}$ on $\mathbb{V}$:
\begin{equation}\label{eqt:linear_map}
\mtx{M}(f) := -\chi_{[-1,1]}\left(\frac{1}{\pi}\int_{\R}f(y)\ln|x-y|\idiff y + c(f)x\right) = \chi_{[-1,1]}\big((-\Delta)^{-1/2}f - c(f)x\big),
\end{equation}
where $\chi_{[-1,1]}$ is the indicator function of the interval $[-1,1]$, and $c(f)$ is given by 
\begin{equation}\label{eqt:cf}
c(f) := -\frac{1}{\pi}\int_{\R}f(y)\ln|1-y|\idiff y = (-\Delta)^{-1/2}f(1).
\end{equation}
Note that for any $f\in \mathbb{V}$, we can calculate that 
\begin{align*}
\mtx{M}(f)(-x) &= -\frac{1}{\pi}\int_{\R}f(y)\ln|x+y|\idiff y + c(f)x\\
 &= \frac{1}{\pi}\int_{\R}f(y)\ln|x-y|\idiff y + c(f)x = -\mtx{M}(f)(x),\quad x\in[-1,1],
\end{align*}
and 
\[\mtx{M}(f)(1) = -\frac{1}{\pi}\int_{\R}f(y)\ln|1-y|\idiff y - c(f) = 0.\]
Also, it is not hard to check that $\|\mtx{M}(f)\|_{H^1}\lesssim \|f\|_{H^1}$. Hence, $\mtx{M}$ maps $\mathbb{V}$ into itself. By these properties, we also find that
\begin{align*}
\int_{\R}\phi_x(x)\ \mtx{M}(f)(x) \idiff x &= \int_{-1}^1\phi_x(x)\ \mtx{M}(f)(x) \idiff x = -\int_{-1}^1\phi(x)\ \partial_x\left((-\Delta)^{-1/2}f - c(f)x\right) \idiff x \\
&= \int_{-1}^1\phi(x)\ (\mtx{H}(f)(x)+c(f)) \idiff x = \int_{\R} \phi(x)\ \chi_{[-1,1]}(\mtx{H}(f)(x)+c(f)) \idiff x,
\end{align*}
for any smooth function $\phi$. Therefore, we have at least in the weak sense that
\begin{equation}\label{eqt:linear_map_derivative}
\partial_x \mtx{M}(f) = - \chi_{[-1,1]}\big( \mtx{H}(f)+c(f) \big).
\end{equation}
This differential relation will be useful in what follows.

Now, finding a solution $(\om, c_l, c_\om)$ that solves the equation \eqref{eqt:eigen_form} reduces to finding an eigen-pair $(\lambda, f)$ that solves the eigenvalue problem
\begin{equation}\label{eqt:eigen_problem}
\lambda f = \mtx{M}(f),\quad \lambda\in \R,\ f\in \mathbb{V}.
\end{equation}
Indeed, $\om = f$ is a solution to \eqref{eqt:main_equation} with $c_l=c_\om = c(f)$.
To show that $\mtx{M}$ admits eigenfunctions in $\mathbb{V}$, we need to prove some nicer properties of $\mtx{M}$.

\begin{lemma} The linear map $\mtx{M}: \mathbb{V}\to \mathbb{V}$ is self-adjoint and positive semidefinite with respect to the inner product $\langle \cdot\,, \cdot \rangle_{\dot{H}^1}$. More precisely, 
\[\langle f, \mtx{M}(g) \rangle_{\dot{H}^1} = \langle \mtx{M}(f), g \rangle_{\dot{H}^1} = \langle f, g \rangle_{\dot{H}^{1/2}(\R)},\quad \forall f,g\in \mathbb{V},\]
and
\[\langle f, \mtx{M}(f) \rangle_{\dot{H}^1} = \|f\|_{\dot{H}^{1/2}(\R)}^2 \geq 0,\quad \forall f\in \mathbb{V}.\]
\end{lemma} 

\begin{proof}
We only consider real-valued functions. The self-adjoint property can be justified by a direct calculation:
\begin{align*}
\langle f, \mtx{M}(g) \rangle_{\dot{H}^1} &= \int_{\R}f_x(x)\ \partial_x \mtx{M}(g)\idiff x = -\int_{-1}^1f_x(x) (\mtx{H}(g)(x)+c(f))\idiff x\\
&= \int_{-1}^1f(x)\partial_x\mtx{H}(g)(x)\idiff x = \langle f, g \rangle_{\dot{H}^{1/2}(\R)} = \langle \mtx{M}(f), g \rangle_{\dot{H}^1}.
\end{align*}
We have used the formula \eqref{eqt:linear_map_derivative}. In particular, 
\[\langle f, \mtx{M}(f) \rangle_{\dot{H}^1} = \langle f, f \rangle_{\dot{H}^{1/2}(\R)} = \|f\|_{\dot{H}^{1/2}(\R)}^2.\]
which proves the positiveness of $\mtx{M}$.
\end{proof}

We can actually further prove that $\mtx{M}$ is positive definite by showing that $\|f\|_{\dot{H}^{1/2}(\R)}\gtrsim \|f\|_{\dot{H}^{1/2}([-1,1])}$. See Theorem \ref{thm:norm_equivalence} in the next subsection.

\begin{lemma}\label{lem:compactness}
The linear map $\mtx{M}: \mathbb{V}\to \mathbb{V}$ is compact. In particular, 
\[\|\mtx{M}(f)\|_{\dot{H}^1}\leq \|f\|_{L^2}\,,\quad \|\mtx{M}(f)\|_{\dot{H}^2([-1,1])}\leq \|f\|_{\dot{H}^1}\,, \]
for all $f\in \mathbb{V}$.
\end{lemma} 

\begin{proof}
The first inequality follows directly from the derivative formula \eqref{eqt:linear_map_derivative}:
\begin{align*}
\|\mtx{M}(f)\|_{\dot{H}^1}^2 &= \int_{-1}^1\left(\mtx{H}(f)(x)^2 + 2c(f)\, \mtx{H}(f)(x) + c(f)^2\right)\idiff x\\
&= \|\mtx{H}(f)\|_{L^2([-1,1])}^2 - 2c(f) \int_{-1}^1\partial_x\mtx{M}(f)\idiff x - 2c(f)^2\\
&= \|\mtx{H}(f)\|_{L^2([-1,1])}^2 - 2c(f)^2\\
&\leq \|f\|_{L^2}^2.
\end{align*}
The second inequality can be proved similarly using the formula \eqref{eqt:linear_map_derivative}:
\[\|\mtx{M}(f)\|_{\dot{H}^2([-1,1])} = \|\mtx{H}(f)\|_{\dot{H}^1([-1,1])}\leq \|\mtx{H}(f)\|_{\dot{H}^1(\R)} = \|f\|_{\dot{H}^1}.\]
The compactness of $\mtx{M}: \mathbb{V}\to \mathbb{V}$ follows immediately.
\end{proof}

By the standard spectral theory of self-adjoint compact operators, we have the following result, which proves the existence of countably infinite solutions in the basic class \eqref{basic_claiss_formal} with the desired regularity.

\begin{theorem}\label{thm:eigenvalue_problem}
The eigenvalue problem \eqref{eqt:eigen_problem} admits countably infinite eigen-pairs $(\lambda_i,f_i)\in \R\times \mathbb{V}$ such that $\lambda_i>0$, $f_i\in H^1_0(\R)\cap H^2([-1,1])\cap C^\infty([-r,r])$ for any $r\in(0,1)$, and $\la f_i\,,\,f_j\ra_{\dot{H}^1}=0, i\neq j$. In particular, the leading eigenfunction $f_*$ corresponding to the largest eigenvalue $\lambda_*$ is a maximizer of the variational problem:
\begin{equation}\label{eqt:variational_form}
\sup_{f\in \mathbb{V}\backslash\{0\}} \frac{\langle f, \mtx{M}(f) \rangle_{\dot{H}^1}}{\|f\|_{\dot{H}^1(\R)}^2} = \sup_{f\in \mathbb{V}\backslash\{0\}} \frac{\|f\|_{\dot{H}^{1/2}(\R)}^2}{\|f\|_{\dot{H}^1(\R)}^2}.
\end{equation}
\end{theorem} 

\begin{proof}
The fact that $\mtx{M}$ has infinite many positive eigenvalues, i.e.\:$\mtx{M}$ is not of finite rank, will be proved in Corollary \ref{cor:eigenvalue_comparison} in the next subsection. In fact, Theorem \ref{thm:norm_equivalence} implies that $0$ is not an eigenvalue of $\mtx{M}$. 

What is yet to be proved is the infinite regularity of an eigenfunction on $[-r,r]$ for any $r\in(0,1)$. Let $(\lambda,f)$ be an eigen-pair of $\mtx{M}$ with $\lambda>0$. We already know that $f\in H^2([-1,1])$, and thus $f\in C^1([-1,1])$. From the derivative formula $\eqref{eqt:linear_map_derivative}$ we find that, on $[-r,r]$, 
\begin{align*}
\partial_x^2f &=  -\lambda^{-1} \mtx{H}(f_x)\\
&= \lambda^{-2} \mtx{H}\big(\chi_{[-1,1]}(\mtx{H}(f) + c(f))\big)\\
&= \lambda^{-2} \left(\mtx{H}(\mtx{H}(f)) - \mtx{H}((1-\chi_{[-1,1]})\mtx{H}(f)) + c(f)\mtx{H}(\chi_{[-1,1]})\right)\\
&=: -\lambda^{-2}f - \lambda^{-2} g_f +  \lambda^{-2}c(f)h_f,
\end{align*}
where 
\[g_f(x) := \mtx{H}\left((1-\chi_{[-1,1]})\mtx{H}(f)\right)(x) = \frac{1}{\pi}\int_{\R\backslash[-1,1]}\frac{\mtx{H}(f)(y)}{x-y}\idiff y\]
and 
\[h_f(x) := \mtx{H}(\chi_{[-1,1]})(x) = \frac{1}{\pi}\ln\left|\frac{x+1}{x-1}\right|\]
are both infinitely differentiable on $[-r,r]$. This inductively implies that 
\[\partial_x^{k+2}f = -\lambda^{-2}\partial_x^kf - \lambda^{-2}\partial_x^kg_f + \lambda^{-2}c(f)\partial_x^kh_f \in C([-r,r])\]
for all $k\geq0$, and thus $f\in C^\infty([-r,r])$.
\end{proof}

As claimed before, we shall give insightful characterizations of the eigenfunctions of $\mtx{M}$. The next theorem provides a powerful point estimate for these eigenfunctions that will be used repeatedly in later sections. In particular, it shows that all zeros of an eigenfunction in $[-1,1]$ are simple. 

\begin{theorem}\label{thm:non-degeneracy}
Let $(\lambda,f)$ be an eigen-pair of $\mtx{M}$ with $\lambda>0$. Then, for any $x\in (-1,1)$, it holds that 
\begin{equation}\label{eqt:point_estimate}
\lambda^2\big[(\partial_xf(x))^2 - 2f(x)\partial_x^2f(x)\big] = f(x)^2 + c(f)^2 + \frac{\lambda}{\pi}\int_{\R} \frac{(f(x)-f(y))^2}{(x-y)^2}\idiff y.
\end{equation}
In particular, for any $r\in [-1,1]$ such that $f(r)=0$, 
\begin{equation}\label{eqt:non-degeneracy}
(\partial_xf(r))^2 = \frac{c(f)^2}{\lambda^2} + \frac{1}{\pi\lambda}\int_{-1}^1 \frac{f(x)^2}{(r-x)^2}\idiff x\geq \frac{c(f)^2}{\lambda^2} + \frac{\lambda}{4\pi}\|f\|_{\dot{H}^1}^2.
\end{equation}
Here in the case of $r= \pm1$, $\partial_xf(r)$ denotes the one-sided derivatives $\partial_{x-}f(1)$ and $\partial_{x+}f(-1)$.
\end{theorem} 

\begin{proof}
With the regularity of $f$ proved in Theorem \ref{thm:eigenvalue_problem}, it is not hard to check that $f^2\in H^2(\R)$. We can thus consider the function $(-\Delta)^{1/2}f^2 = \mtx{H}(\partial_x(f^2))$ at every single point without worrying about infinitely large values. We apply the formula $\eqref{eqt:linear_map_derivative}$ to derive the following on $[-1,1]$:
\begin{equation}\label{eqt:point_estimate_midstep}
\begin{split}
\mtx{H}(\partial_x(f^2)) &= 2\mtx{H}(f\partial_xf) = -2\lambda^{-1}\mtx{H}\big(f(\mtx{H}(f)+c(f))\big)\\
&= -\lambda^{-1}\big(\mtx{H}(f)^2 - f^2 + 2c(f)\mtx{H}(f)\big)= -\lambda (\partial_xf)^2 + \lambda^{-1}\big(f^2 + c(f)^2\big).
\end{split}
\end{equation}
Here we have used Tricomi's identity that $\mtx{H}(g\mtx{H}(g)) = (\mtx{H}(g)^2 - g^2)/2$ for any $g\in L^2(\R)$. On the other hand, for $x\in (-1,1)$, 
\begin{align*}
\mtx{H}(\partial_x(f^2))(x) &= (-\Delta)^{1/2}(f^2)(x) \\
&= \frac{1}{\pi}\text{P.V.}\int_{\R}\frac{f(x)^2-f(y)^2}{(x-y)^2}\idiff y \\
&= 2f(x)\cdot \frac{1}{\pi}\text{P.V.}\int_{\R}\frac{f(x)-f(y)}{(x-y)^2}\idiff y - \frac{1}{\pi}\int_{\R}\frac{(f(x)-f(y))^2}{(x-y)^2}\idiff y \\
&= 2f(x)\, (-\Delta)^{1/2}f(x) - \frac{1}{\pi}\int_{\R}\frac{(f(x)-f(y))^2}{(x-y)^2}\idiff y.
\end{align*}
It again follows from the derivative formula $\eqref{eqt:linear_map_derivative}$ that
\[(-\Delta)^{1/2}f(x) = \partial_x\mtx{H}(f)(x) = -\lambda\partial_x^2f(x), \quad x\in (-1,1).\]
Combining the calculations above yields \eqref{eqt:point_estimate}. 

Furthermore, for any $r\in [-1,1]$ such that $f(r)=0$, we have
\[\mtx{H}(\partial_x(f^2))(r) = (-\Delta)^{1/2}(f^2)(r) = \frac{1}{\pi}\text{P.V.}\int_{\R}\frac{f(r)^2-f(x)^2}{(r-x)^2}\idiff x = -\frac{1}{\pi}\int_{-1}^1\frac{f(x)^2}{(r-x)^2}\idiff x.\]
It then follows from \eqref{eqt:point_estimate_midstep} and Lemma \ref{lem:compactness} that 
\begin{equation}
(\partial_xf(r))^2 - \frac{c(f)^2}{\lambda^2} = \frac{1}{\pi\lambda}\int_{-1}^1 \frac{f(x)^2}{(r-x)^2}\idiff x \geq \frac{1}{4\pi\lambda}\|f\|_{L^2}^2\geq \frac{1}{4\pi\lambda} \|\mtx{M}(f)\|_{\dot{H}^1}^2 = \frac{\lambda}{4\pi} \|f\|_{\dot{H}^1}^2,
\end{equation}
which is \eqref{eqt:non-degeneracy}. 
\end{proof}

An immediate consequence of Theorem \ref{thm:non-degeneracy} is that all eigenvalues of $\mtx{M}$ are simple.

\begin{corollary}
For each positive eigenvalue $\lambda$ of $\mtx{M}$, the corresponding non-trivial eigenfunction is unique up to a multiplicative constant. 
\end{corollary}

\begin{proof}
If a positive eigenvalue $\lambda$ corresponds to two linearly independent eigenfunctions $f,\hat{f}\in\mathbb{V}$,  then $h = f + C \hat{f}$ is also an non-trivial eigenfunction of $\lambda$ for any constant $C\in \R$. However, by Theorem \ref{thm:non-degeneracy}, $f_x(0)$ and $\hat{f}_x(0)$ are both nonzero since $f(0)=\hat f(0)=0$. We can thus find some constant $C$ such that $h_x(0) =  f_x(0) + C \hat{f}_x(0) = 0$, which leads to a contradiction.
\end{proof}

The next corollary upper bounds the distance between two neighboring zeros of an eigenfunction in terms of the eigenvalue. In particular, it characterizes how oscillatory an eigenfunction can be when the corresponding eigenvalue is small.

\begin{corollary}\label{cor:root_distance}
Let $(\lambda,f)$ be an eigen-pair of $\mtx{M}$ with $\lambda>0$. Then the distance between any two neighboring zeros of $f$ in $[0,1]$ must be smaller than or equal to $2\pi\lambda$. As a consequence, $f$ has at least $\lceil\frac{1}{2\pi\lambda}\rceil-1$ zeros in $(0,1)$, provided that $2\pi\lambda<1$.
\end{corollary}

\begin{proof}
The $H^2$ regularity of $f$ and \eqref{eqt:non-degeneracy} implies that the zeros of $f$ in $[-1,1]$ are discrete. Let $r_1,r_2\in[0,1]$, $r_1<r_2$, be two neighboring zeros of $f$. Since $f\in H^2([-1,1])$, we may assume that $f>0$ on $(r_1,r_2)$, and that $f$ attains its local maximum at $x_*\in(r_1,r_2)$.

We can rewrite the equality \eqref{eqt:point_estimate} in Theorem \ref{thm:non-degeneracy} as
\[2\lambda^2\left((\partial_xf(x))^2 - f(x)\partial_x^2f(x)  \right) = \lambda^2(\partial_xf(x))^2 + f(x)^2 + c(f)^2 + \frac{\lambda}{\pi}\int_{\R}\frac{(f(x)-f(y))^2}{(x-y)^2}\idiff y,\]
which gives
\begin{equation}\label{eqt:log-concave}
\partial_x\left(\frac{\partial_xf(x)}{f(x)}\right) \leq -\frac{1}{2\lambda^2} - \frac{1}{2}\left(\frac{\partial_xf(x)}{f(x)}\right)^2,\quad x\in (r_1,r_2).
\end{equation}
Denote $g:= \lambda\partial_xf / f$. Then it satisfies 
\[\partial_xg \leq -\frac{1}{2\lambda}(1+g^2),\quad g(x_*) = 0.\]
We solve it to find 
\[\lambda\left|\frac{\partial_xf(x)}{f(x)}\right|\geq \tan\left(\frac{|x-x_*|}{2\lambda}\right),\quad x\in (r_1,r_2).\]
Since $f\in H^2([-1,1])$ and $f>0$ on $(r_1,r_2)$, the right-hand side of the inequality above should stay finite on any compact subset of $(r_1,r_2)$. This implies that 
\[\frac{1}{2\lambda}\max\{|r_1-x_*|,|r_2-x_*|\}\leq \frac{\pi}{2},\]
and thus $|r_2-r_1|\leq 2\pi\lambda$.
\end{proof}

From the inequality \eqref{eqt:log-concave} we also learn that $|f|$ is log-concave, i.e.\ $\ln|f|$ is concave, between any two neighboring zeros of $f$ in $[-1,1]$. As a result, there is only one local maximum or local minimum between any two neighboring zeros of an eigenfunction of $\mtx{M}$. \\

Finally, we show that, for any eigenfunction $f$ of $\mtx{M}$, the corresponding $c(f)$ is non-zero. That is, each eigenfunction $f$ can be rescaled (as in \eqref{eqt:scaling}) to become a solution of \eqref{eqt:main_equation} with $c_l=c_\om = -1$.\\

\begin{theorem}\label{thm:cf_nonzero} Let $(\lambda,f)$ be an eigen-pair of $\mtx{M}$ with $\lambda>0$. Let $c(f)$ be defined in \eqref{eqt:cf}. Then $c(f)\neq 0$. 
\end{theorem}

\begin{proof} We prove the claim by contradiction. Suppose that $c(f) = 0$.
Define a new function $\phi$ as
\[\phi(x) = \int_{-1}^xf(y)\idiff y, \quad x\in [-1,1];\quad \phi(x) = 0,\quad x\in \R\backslash[-1,1].\]
Clearly, $\partial_x\phi(x) = f(x)$ and $\phi(x) = \phi(-x)$ for all $x\in \R$. It follows from \eqref{eqt:linear_map_derivative} that, for $x\in[-1,1]$,
\[\lambda\partial_xf(x) = -\mtx{H}(f)(x) = -\partial_x\mtx{H}(\phi)(x).\]
Hence, there is some constant $c_1$ such that $-\lambda f = \mtx{H}(\phi) + c_1$ on $[-1,1]$. But since $f$ and $\mtx{H}(\phi)$ are both odd functions of $x$, we have $c_1=0$. We thus obtain a formula for $\phi$ that is similar to \eqref{eqt:linear_map_derivative}:
\[\lambda\partial_x\phi = -\chi_{[-1,1]}\mtx{H}(\phi).\]
Recall how we derive \eqref{eqt:non-degeneracy} from \eqref{eqt:linear_map_derivative}. Using the same techniques as in the proof of Theorem \ref{thm:non-degeneracy}, we can similarly show that, for any $r\in[-1,1]$ such that $\phi(r)=0$,  
\[(\partial_x\phi(r))^2 = \frac{1}{\pi\lambda}\int_{-1}^1 \frac{\phi(x)^2}{(r-x)^2}\idiff x.\]
Note that $\phi(1) = \int_{-1}^1f(y)\idiff y = 0$ since $f$ is odd, and $\partial_x\phi(1) = f(1) = 0$. This implies that 
\[\frac{1}{\pi\lambda}\int_{-1}^1 \frac{\phi(x)^2}{(1-x)^2}\idiff x = 0,\]
which is impossible unless $\phi(x)\equiv 0$ on $[-1,1]$. But this contradicts the fact that $f$ is non-trivial on $[-1,1]$. Therefore, we must have $c(f)\neq0$.
\end{proof}

In the next two subsections, we will prove the claimed sign property of the solutions in the basic class \eqref{basic_claiss_formal}. More precisely, we will show that the leading eigenfunction of $\mtx{M}$ corresponding to the largest eigenvalue is the only eigenfunction that is strictly positive on $(0,1)$ (up to a multiplicative constant), while all the other eigenfunctions are sign-changing on $(0,1)$. Guided by Corollary \ref{cor:root_distance}, we need to estimate the eigenvalues of $\mtx{M}$.

\subsection{Comparison with a classical eigenvalue problem}
One may easily relate the variational problem \eqref{eqt:variational_form} to a more classic and local one: 
\begin{equation}\label{eqt:variational_form_local}
\sup_{f\in \mathbb{V}\backslash\{0\}} \frac{\la f, (-\widetilde\Delta)^{-1/2}f \ra_{\dot{H}^1}}{\|f\|_{\dot{H}^1([-1,1])}^2} = \sup_{f\in \mathbb{V}\backslash\{0\}} \frac{\|f\|_{\dot{H}^{1/2}([-1,1])}^2}{\|f\|_{\dot{H}^1([-1,1])}^2}.
\end{equation}
Here $(-\widetilde\Delta)^{-1}$ denotes the inverse Laplacian on the bounded interval $[-1,1]$ associated with zero Dirichlet boundary conditions, and $(-\widetilde\Delta)^{-1/2}$ is defined as the square root of it. One should see $f$ in the expression $(-\widetilde\Delta)^{-1/2}f$ as a function in $H^1_0([-1,1])$. It is well-known that all eigen-pairs of $(-\widetilde\Delta)^{-1/2}$ corresponding to this variational problem are given by 
\begin{equation}\label{eqt:classic_eigenpair}
\tilde \lambda_n = \frac{1}{n\pi},\quad \tilde f_n = \chi_{[-1,1]}\frac{\sin(n\pi x)}{n\pi},\quad n \in \Z_+. 
\end{equation}
Here each $\tilde f_n$ is normalized so that $\|\tilde f_n\|_{\dot{H}^1([-1,1])} = 1$. By the similarity between these two variational problems, it is conceivable that the eigen-pairs of the eigenvalue problem \eqref{eqt:eigen_problem} are similar to those in \eqref{eqt:classic_eigenpair}. In fact, we have numerically computed the first few eigenfunctions of $\mtx{M}$ and plot them in Figure \ref{fig:Eigenfunctions}. One can see that they are quite similar to the corresponding eigenfunctions of $(-\widetilde\Delta)^{-1/2}$. In particular, it is observed that $f_n$, the eigenfunction of $\mtx{M}$ corresponding to its $n$-th largest eigenvalue, has the same number of zeros on $[0,1]$ as $\tilde{f}_n$ with similar locations of the zeros. Besides, both $|f_n|$ and $|\tilde{f}_n|$ are log-concave between their neighboring zeros (for $|f_n|$, this is proved in \eqref{eqt:log-concave}). This explains the similarity in the shapes of $f_n$ and $\tilde{f}_n$.

As for the eigenvalues, we can show that $\lambda_n$ is comparable to $\tilde\lambda_n$ for all $n$. To achieve this, we first prove that the two semi-norms $\|\cdot \|_{\dot{H}^{1/2}(\R)}$ and $\|\cdot\|_{\dot{H}^{1/2}([-1,1])}$ are equivalent over the space $\mathbb{V}$.

\begin{theorem}\label{thm:norm_equivalence}
For any $f\in \mathbb{V}$, it holds that 
\begin{equation}\label{eqt:norm_equivalence}
\frac{\sqrt{2}}{\pi}\, \|f\|_{\dot{H}^{1/2}([-1,1])}\leq \|f\|_{\dot{H}^{1/2}(\R)} \leq \|f\|_{\dot{H}^{1/2}([-1,1])}.
\end{equation}
\end{theorem} 

\begin{proof}
We first prove the second inequality in \eqref{eqt:norm_equivalence}. By Lemma \ref{lem:compactness} and the definition of $(-\widetilde\Delta)^{-1}$, we have that 
\[\la f, \mtx{M}^2(f) \ra_{\dot{H}^1}\leq \|f\|_{L^2([-1,1])}^2 = \la f, (-\widetilde\Delta)^{-1}f \ra_{\dot{H}^1}.\]
Since the square-root function is operator-monotone increasing over positive semidefinite, self-adjoint and compact operators (which is a consequence of the L\"owner--Heinz Theorem; e.g., see \cite[Theorem 2.6]{carlen2010trace}), it immediately follows that
\[\|f\|_{\dot{H}^{1/2}(\R)}^2 = \la f, \mtx{M}(f) \ra_{\dot{H}^1}\leq \la f, (-\widetilde\Delta)^{-1/2}f \ra_{\dot{H}^1} = \|f\|_{\dot{H}^{1/2}([-1,1])}^2.\]

Next, we prove the first inequality in \eqref{eqt:norm_equivalence} using the double-integral expressions of the two semi-norms. Owing to the odd symmetry of $f\in \mathbb{V}$, we have
\begin{align*}
\|f\|_{\dot{H}^{1/2}([-1,1])}^2 &= \frac{1}{2\pi}\int_{-1}^1\int_{-1}^1 \frac{(f(x)-f(y))^2}{\frac{4}{\pi^2}\sin^2(\frac{\pi}{2}(x-y))}\idiff x\idiff y\\
&= \frac{\pi}{4}\int_{0}^1\int_{0}^1 \frac{(f(x)-f(y))^2}{\sin^2(\frac{\pi}{2}(x-y))}\idiff x\idiff y + \frac{\pi}{4}\int_{0}^1\int_{0}^1 \frac{(f(x)+f(y))^2}{\sin^2(\frac{\pi}{2}(x+y))}\idiff x\idiff y\\
&=: I_1 + I_2.
\end{align*}
Since $|t/\sin (t)|\leq \pi/2$ for $t\in[-\pi/2,\pi/2]$, the first term above can be controlled as
\[I_1 \leq \frac{\pi}{4}\int_{0}^1\int_{0}^1 \frac{(f(x)-f(y))^2}{(x-y)^2}\idiff x\idiff y.\]
To estimate the second term $I_2$, we further decompose the integral as
\begin{align*}
I_2 &= \frac{\pi}{4}\int_{0}^1\int_{0}^1 \chi_{\{|x+y|\leq1\}}\frac{(f(x)+f(y))^2}{\sin^2(\frac{\pi}{2}(x+y))}\idiff x\idiff y + \frac{\pi}{4}\int_{0}^1\int_{0}^1 \chi_{\{|x+y|>1\}}\frac{(f(x)+f(y))^2}{\sin^2(\frac{\pi}{2}(x+y))}\idiff x\idiff y\\
&=: I_{2,1} + I_{2,2}.
\end{align*}
We bound the denominator in $I_{2,1}$ in a similar way to get 
\[I_{2,1}\leq \frac{\pi}{4}\int_{0}^1\int_{0}^1 \chi_{\{|x+y|\leq1\}}\frac{(f(x)+f(y))^2}{(x+y)^2}\idiff x\idiff y \leq \frac{\pi}{4}\int_{0}^1\int_{0}^1 \frac{(f(x)+f(y))^2}{(x+y)^2}\idiff x\idiff y.\]
The term $I_{2,2}$ is handled similar but with change of variables in $x$ and $y$:
\begin{align*}
I_{2,2} &= \frac{\pi}{4}\int_{0}^1\int_{0}^1 \chi_{\{|x+y|>1\}}\frac{(f(x)+f(y))^2}{\sin^2(\frac{\pi}{2}(2 - x - y))}\idiff x\idiff y\\
&\leq \frac{\pi}{4}\int_{0}^1\int_{0}^1 \chi_{\{|x+y|>1\}}\frac{(f(x)+f(y))^2}{(2 - x - y)^2}\idiff x\idiff y\\
&\leq \frac{\pi}{2}\int_{0}^1\int_{0}^1 \frac{f(x)^2+f(y)^2}{(2 - x - y)^2}\idiff x\idiff y\\
&= \frac{\pi}{2}\int_{0}^1\left(\int_{1}^2 \frac{f(x)^2}{(x-y)^2}\idiff y\right)\idiff x + \frac{\pi}{2}\int_{0}^1\left(\int_{1}^2 \frac{f(y)^2}{(x-y)^2}\idiff x\right)\idiff y\\
&= \frac{\pi}{2}\int_{0}^1\left(\int_{1}^2 \frac{(f(x)-f(y))^2}{(x-y)^2}\idiff y\right)\idiff x + \frac{\pi}{2}\int_{0}^1\left(\int_{1}^2 \frac{(f(x)-f(y))^2}{(x-y)^2}\idiff x\right)\idiff y.
\end{align*}
Combining the preceding estimates, we obtain that
\begin{align*}
\|f\|_{\dot{H}^{1/2}([-1,1])}^2 &\leq \frac{\pi}{2}\int_{0}^{+\infty}\int_{0}^{+\infty}\left(\frac{(f(x)-f(y))^2}{(x-y)^2} + \frac{(f(x)+f(y))^2}{(x+y)^2} \right) \idiff x \idiff y\\
&= \frac{\pi}{4}\int_{\R^2}\frac{(f(x)-f(y))^2}{(x-y)^2} \idiff x \idiff y\\
&= \frac{\pi^2}{2}\|f\|_{\dot{H}^{1/2}(\R)}^2,
\end{align*}
which is the desired inequality.
\end{proof}

We remark that the factor $\sqrt{2}/\pi$ can be slightly improved by choosing a finer partition of $I_2$ in the proof above, instead of using $\chi_{\{|x+y|\leq1\}}$. As an immediate consequence of Theorem \ref{thm:norm_equivalence}, the sorted eigenvalues of $\mtx{M}$ and $(-\widetilde\Delta)^{-1/2}$ are pairwise comparable.

\begin{corollary}\label{cor:eigenvalue_comparison}
For each $n\geq 1$, 
\begin{equation}\label{eqt:eigenvalue_comparison}
\frac{2}{\pi^2}\tilde\lambda_n \leq \lambda_n < \tilde\lambda_n,
\end{equation}
where $\lambda_n$ and $\tilde\lambda_n=\frac{1}{n\pi}$ are the $n$-th largest eigenvalues of $\mtx{M}$ and $(-\widetilde\Delta)^{-1/2}$, respectively. 
\end{corollary}

\begin{proof}
Recall the Courant--Fischer--Weyl max-min principle for a general self-adjoint compact operator $\mtx{A}$ acting on $\mathbb{V}$:
\[\lambda_n(\mtx{A}) = \sup_{\mathbb{U}\subset\mathbb{V}, \dim(\mathbb{U}) = n}\,\inf_{g\in \mathbb{U}\backslash\{0\}}\frac{\la g, \mtx{A}(g)\ra_{\dot{H}^1}}{\|g\|_{\dot{H}^1}^2},\]
where $\lambda_n(\mtx{A})$ denotes the $n$-th largest eigenvalue of $\mtx{A}$. This and Theorem \ref{thm:norm_equivalence} together imply that 
\[\frac{2}{\pi^2}\tilde\lambda_n \leq \lambda_n \leq \tilde\lambda_n.\]
We need to show that the second inequality above is strict for all $n$.

For each $n\geq 1$, let $f_n$ be the unique eigenfunction of $\mtx{M}$ corresponding to $\lambda_n$ with $\|f_n\|_{\dot{H}^1} = 1$, and let $\tilde{f}_n$ be defined as in \eqref{eqt:classic_eigenpair}. Let $\mathbb{V}_n = \mathrm{span}\{f_1,f_2,\dots,f_n\}$. Suppose that $\lambda_n=\tilde{\lambda}_n$. By Theorem \ref{thm:norm_equivalence}, for any $g = \sum_{i=1}^nc_if_i\in \mathbb{V}$ with $\|g\|_{\dot{H}^1}^2 = \sum_{i=1}^nc_i^2 = 1$,
\[\|g\|_{\dot{H}^{1/2}([-1,1])}^2 \geq \|g\|_{\dot{H}^{1/2}(\R)}^2 = \suml_{i=1}^nc_i^2\lambda_i \geq \lambda_n = \tilde{\lambda}_n.\]
Therefore, by the Courant--Fischer--Weyl max-min principle,
\[\tilde{\lambda}_n \leq \inf_{g\in \mathbb{V}_n\backslash\{0\}}\frac{\|g\|_{\dot{H}^{1/2}([-1,1])}^2}{\|g\|_{\dot{H}^1}^2}\leq  \sup_{\mathbb{U}\subset\mathbb{V}, \dim(\mathbb{U}) = n}\,\inf_{g\in \mathbb{U}\backslash\{0\}}\frac{\|g\|_{\dot{H}^{1/2}([-1,1])}^2}{\|g\|_{\dot{H}^1}^2}= \tilde{\lambda}_n.\]
That is, 
\[\tilde{\lambda}_n = \inf_{g\in \mathbb{V}_n\backslash\{0\}}\frac{\|g\|_{\dot{H}^{1/2}([-1,1])}^2}{\|g\|_{\dot{H}^1}^2},\]
which implies that $\tilde{f}_n\in \mathbb{V}_n$. Then, from the proof of Lemma \ref{lem:compactness}, we find that 
\[\tilde{\lambda}_n^2 = \lambda_n^2 \leq \|\mtx{M}(\tilde{f}_n)\|_{\dot{H}^1}^2 = \|\mtx{H}(\tilde{f}_n)\|_{L^2([-1,1])}^2 - 2c(\tilde{f}_n)^2\leq  \|\mtx{H}(\tilde{f}_n)\|_{L^2(\R)}^2 = \|\tilde{f}_n\|_{L^2}^2 = \tilde{\lambda}_n^2.\]
This means the inequalities above must all be equality. In particular, $\|\mtx{H}(\tilde{f}_n)\|_{L^2([-1,1])}^2 = \|\mtx{H}(\tilde{f}_n)\|_{L^2(\R)}^2$. However, it is straightforward to check that
\[\|\mtx{H}(\tilde{f}_n)\|_{L^2(\R)}^2 - \|\mtx{H}(\tilde{f}_n)\|_{L^2([-1,1])}^2 = \|\mtx{H}(\tilde{f}_n)\|_{L^2(\R\backslash[-1,1])}^2>0,\]
which leads to a contradiction. In fact, for any $x>1$, 
\[(-1)^n\,\mtx{H}(\tilde{f}_n)(x) = \frac{(-1)^n}{n\pi^2}\int_{-1}^1\frac{\sin(n\pi y)}{x-y}\idiff y =  \frac{-1}{n^3\pi^2}\suml_{k=1}^n\int_0^1\frac{\sin(\pi s)\idiff s}{(x-\frac{s+2k-n-2}{n})(x-\frac{s+2k-n-1}{n})} <0.\]
Therefore, we must have $\lambda_n<\tilde{\lambda}_n$.
\end{proof}

We are now able to prove the sign-changing property for all eigenfunctions of $\mtx{M}$ except for the leading one, using the second inequality in \eqref{eqt:eigenvalue_comparison}.

\begin{corollary}\label{cor:sign-changing}
Let $f_n$ be a non-trivial eigenfunction of $\mtx{M}$ corresponding to the $n$-th largest eigenvalue $\lambda_n$. Then $f_n$ has at least $\lfloor\frac{n}{2}\rfloor$ zeros in $(0,1)$. As a result, $f_n$ must be sign-changing on $(0,1)$ for all $n\geq 2$.
\end{corollary}

\begin{proof}
It follows from Corollary \ref{cor:eigenvalue_comparison} that $\frac{1}{2\pi\lambda_n}>\frac{n}{2}$. Corollary \ref{cor:root_distance} then implies that the number of zeros of $f_n$ in $(0,1)$ can be bounded from below by $\lceil\frac{1}{2\pi\lambda_n}\rceil - 1 \geq \lfloor\frac{n}{2}\rfloor$. Therefore, for each $n\geq 2$, $f_n$ has at least one zero in $(0,1)$. Since all zeros of $f_n$ in $(-1,1)$ are simple (see Theorem \ref{thm:non-degeneracy}), $f_n$ must be sign-changing on $(0,1)$.
\end{proof}

We remark that Corollary \ref{cor:sign-changing} only gives a crude lower bound on the numbers of zeros of $f_n$ that has the correct order (namely $O(n)$). Based on the comparison of the variational problems of $\mtx{M}$ and $(-\widetilde{\Delta})^{-1/2}$, and also based on our numerical observations (see Appendix \ref{sec:numerical}), it is reasonable to conjecture that the eigenfunction of $\mtx{M}$ corresponding to its $n$-th largest eigenvalue actually has exactly $n-1$ zeros in $(0,1)$ (as many as the zeros of $sin(n\pi x)$ in $(0,1)$). Proving this conjecture may require the establishment of a variant of the Sturm--Picone comparison theorem adapted to the operator $\mtx{M}$. However, we have not been able to achieve this with our current approach.

\subsection*{An ``illegal'' eigenfunction with eigenvalue $0$} It is well-known that $0$ does not belong to the point spectrum of $(-\widetilde\Delta)^{-1/2}\in \mathcal{L}(\mathbb{V})$. Hence, by the first inequality in \eqref{eqt:norm_equivalence}, $0$ is not an eigenvalue of $\mtx{M}\in\mathcal{L}(\mathbb{V})$ either. However, it is possible to find a non-trivial ``eigenfunction'' of $\mtx{M}$ associated with the eigenvalue $0$ in a larger function space. In fact, Castro \cite{martinez2010nonlinear} has constructed such a function, 
\[\Omega_0(x) = -\chi_{[-1,1]}\frac{x}{\sqrt{1-x^2}}\notin\mathbb{V},\]
which verifies the eigenvalue relation $\mtx{M}(\Omega_0) = 0$. To see this, one can calculate that 
\[(-\Delta)^{-1/2}\Omega_0(x) = 
\begin{cases}
-x+\sqrt{x^2-1}, & x>1,\\
-x, & x\in [-1,1], \\
-x-\sqrt{x^2-1}, & x<-1,
\end{cases}\]
and $c(\Omega_0) = -1$. Therefore, $\mtx{M}(\Omega_0) = \chi_{[-1,1]}\big((-\Delta)^{-1/2}\Omega_0 + x\big) \equiv 0$. Furthermore, this function $\Omega_0$ is a self-similar profile of the generalized Constantin--Lax--Majda equation \eqref{eqt:gCLM} for any $a$ with $c_\om=-1$ and $c_l=-a$. That is, 
\[\om_a(x,t) := \frac{1}{T-t}\cdot\Omega_0\left(\frac{x}{(T-t)^{-a}}\right)\]
is a self-similar solution to \eqref{eqt:gCLM} for all values of $a$. Nevertheless, the inherent low regularity of $\Omega_0$ makes it unuseful in proving finite-time singularity of \eqref{eqt:gCLM} from smooth initial data with finite energy.

\subsection{On the leading eigenfunction} 
In this subsection, we complete the proof of the first part of Theorem \ref{thm:main_theorem} (about the basic class) by showing the following property of the leading eigenfunction $f_*$ of $\mtx{M}$.

\begin{theorem}\label{thm:uniqueness_positiveness}
Let $f_*$ be the leading eigenfunction of the eigenvalue problem \eqref{eqt:eigen_problem} corresponding to the largest eigenvalue. Then $f_*$ is strictly positive on $(0,1)$ up to a multiplicative constant.
\end{theorem} 

Some preparations are needed before we can prove this theorem. \\

It is well-known that the $H^{1/2}$-seminorm on $\R$ can be calculated by extending the function to a higher-dimensional space. Given $f\in H^{1/2}(\R)$, we let $g_f$ solve the following Dirichlet problem in the upper-half plane $\R^2_+$
\[ -\Delta g_f = 0,\quad g_f |_{y = 0} = f. \]
Then
\[ \|f\|_{\dot{H}^{1/2}(\R)} = \|\nabla g_f\|_{L^2(\R^2_+)}. \]
In what follows, for a given function $f$ on $\R$, we will always use the notation $g_f$ to denote its harmonic extension to $\R_+^2$. We first use this technique to prove a decomposition inequality of $f_*$.

\begin{lemma}\label{lem:H1norm_decomposition}
Assume that $f_*$ can be decomposed as $f_* = f_1 + f_2$ such that each $f_i\in \mathbb{V}$ is not identically zero, and the intersection of their supports has Lebesgue measure zero. Then 
\[\frac{\|f_1\|_{\dot{H}^{1/2}(\R)}^2}{\|f_1\|_{\dot{H}^1(\R)}^2} + \frac{\|f_2\|_{\dot{H}^{1/2}(\R)}^2}{\|f_2\|_{\dot{H}^1(\R)}^2} \geq \frac{\|f_*\|_{\dot{H}^{1/2}(\R)}^2}{\|f_*\|_{\dot{H}^1(\R)}^2}.\]
\end{lemma}

\begin{proof}
Let $f_t:= f_1 + tf_2$. By the non-overlapping assumption on $f_i$, 
\begin{equation}\label{eqt:H1norm_decomposition}
\|f_t\|_{\dot{H}^1}^2 = \|f_1\|_{\dot{H}^1}^2 + t^2\|f_2\|_{\dot{H}^1}^2,
\end{equation}
so 
\[\left.\frac{\diff\ }{\diff t}\right|_{t=1}\|f_t\|_{\dot{H}^1}^2 = 2\|f_2\|_{\dot{H}^1}^2.\]
On the other hand,
\begin{align*}
\|f_t\|_{\dot{H}^{1/2}}^2 &= \int_{\R_+^2} \big|\nabla g_{f_1} + t\nabla g_{f_2}\big|^2\idiff x \idiff y\\
&= \int_{\R_+^2} \big|\nabla g_{f_1}\big|^2 + t^2\big|\nabla g_{f_2}\big|^2 + 2t \nabla g_{f_1}\cdot \nabla g_{f_2}\idiff x \idiff y.
\end{align*}
Hence,
\begin{align*}
\left.\frac{\diff\ }{\diff t}\right|_{t = 1}\|f_t\|_{\dot{H}^{1/2}}^2
&=  \int_{\R_+^2} 2\big|\nabla g_{f_2}\big|^2 + 2\nabla g_{f_1}\cdot \nabla g_{f_2}\idiff x \idiff y\\
&=  \int_{\R_+^2} \big|\nabla g_{f_2}\big|^2
- \big|\nabla g_{f_1}\big|^2
+ \big|\nabla g_{f_1}+\nabla g_{f_2}\big|^2\idiff x \idiff y\\
&=  \|f_2\|_{\dot{H}^{1/2}}^2 - \|f_1\|_{\dot{H}^{1/2}}^2
+ \|f_*\|_{\dot{H}^{1/2}}^2.
\end{align*}
Combining these calculations, we find that
\[
\left.\frac{\diff\ }{\diff t}\right|_{t = 1} \frac{\|f_t\|_{\dot{H}^{1/2}}^2}{\|f_t\|_{\dot{H}^{1}}^2}
= \frac{\|f_2\|_{\dot{H}^{1/2}}^2 - \|f_1\|_{\dot{H}^{1/2}}^2
+ \|f_*\|_{\dot{H}^{1/2}}^2}{\|f_*\|_{\dot{H}^1}^2} - \frac{\|f_*\|_{\dot{H}^{1/2}}^2\cdot 2 \|f_2 \|_{\dot{H}^1}^2}{\|f_*\|_{\dot{H}^1}^4}.
\]
By the optimality of $f_*$ in the variational problem \eqref{eqt:variational_form}, the left-hand side has to be $0$, which gives
\[
\frac{\|f_1\|_{\dot{H}^{1/2}}^2 - \|f_2\|_{\dot{H}^{1/2}}^2}{\|f_*\|_{\dot{H}^{1/2}}^2}= 1 - \frac{2\|f_2 \|_{\dot{H}^1}^2}{\|f_*\|_{\dot{H}^1}^2}.
\]
Denote
\[
\alpha_1 : = \frac{\|f_1\|_{\dot{H}^{1/2}}^2}{\|f_*\|_{\dot{H}^{1/2}}^2},\quad
\alpha_2 : = \frac{\|f_2\|_{\dot{H}^{1/2}}^2}{\|f_*\|_{\dot{H}^{1/2}}^2},\quad
\theta : = \frac{\|f_2\|_{\dot{H}^1}^2}{\|f_*\|_{\dot{H}^1}^2}.
\]
Then the above equation becomes
$\alpha_1 -\alpha_2 = 1-2\theta$.
On the other hand, the triangle inequality gives
$\sqrt{\alpha_1} + \sqrt{\alpha_2}\geq 1$.
Combining them yields
$\sqrt{\alpha_2+1-2\theta} + \sqrt{\alpha_2}\geq 1$,
which implies
$\alpha_2\geq \theta^2$, and thus $\alpha_1 = \alpha_2 + 1-2\theta \geq (1-\theta)^2$.
Therefore,
\[
\frac{\alpha_1}{1-\theta} + \frac{\alpha_2}{\theta} \geq  1.
\]
This is exactly the desired inequality thanks to \eqref{eqt:H1norm_decomposition} with $t=1$. 
\end{proof}

Next, we define two operators that manipulate the supports of functions. The first one is a ``translation towards the origin'' operator $\mtx{T}$ defined as follows: let $f$ be a non-zero odd function on $\R$, and denote 
\[l_f := \sup\, \{l>0:\, f \equiv0 \text{ on } [-l,l]\}.\] 
Then define 
\[\mtx{T}(f)(x) := \left\{
\begin{array}{ll}
f(x+l_f),& \text{if } x\geq 0,\\
f(x-l_f),& \text{if } x< 0.
\end{array}
\right.\]
It is straightforward to verify that, for $f\in \mathbb{V}$ that does not change sign on $[0,1]$, $\mtx{T}(f)\in \mathbb{V}$, and
\[\|\mtx{T}(f)\|_{\dot{H}^1(\R)} = \|f\|_{\dot{H}^1(\R)},\quad \|\mtx{T}(f)\|_{\dot{H}^{1/2}(\R)} \geq \|f\|_{\dot{H}^{1/2}(\R)},\]
and, in particular, if $l_f>0$ then 
\[\|\mtx{T}(f)\|_{\dot{H}^{1/2}(\R)} > \|f\|_{\dot{H}^{1/2}(\R)}.\]
The inequality follows from the explicit formula of the $H^{1/2}$-seminorm:
\[\|f\|_{\dot{H}^{1/2}(\R)}^2 = \frac{1}{2\pi}\int_{\R^2}\left(\frac{f(x)-f(y)}{x-y}\right)^2 \idiff x \idiff y.\]

The second operator $\mtx{S}$ rescales the support of an odd function to $[-1,1]$. Let $f$ be a non-zero, odd, and compactly supported function on $\R$, and denote 
\[L_f := \inf\, \{L>0:\, f \equiv0 \text{ on } (L,+\infty)\}.\]
Define
\[\mtx{S}(f)(x) = f(L_f\,x).\] 
It is clear that, for a non-zero $f\in \mathbb{V}$, $\mtx{S}(f)\in \mathbb{V}$, and 
\[\|\mtx{S}(f)\|_{\dot{H}^{1/2}(\R)}^2 =  \|f\|_{\dot{H}^{1/2}(\R)}^2,\quad \|\mtx{S}(f)\|_{\dot{H}^1(\R)}^2 =  L_f\|f\|_{\dot{H}^1(\R)}^2.\]

We are now ready to prove Theorem \ref{thm:uniqueness_positiveness}.

\begin{proof}[Proof of Theorem \ref{thm:uniqueness_positiveness}]
We prove by contradiction that $f_*$ is strictly positive on $(0,1)$ up to a multiplicative constant. Suppose that this is not true. Then since $f_*\in C(\R)\cup C^1([-1,1])$, $f_*$ must have at least one zero in $(0,1)$. Define 
\[r := \sup\{x\in(0,1): f_*(x) = 0\} > 0.\]
By Lemma \ref{thm:non-degeneracy}, $\partial_{x-}f_*(1)\neq 0$, so $r<1$. We may thus assume that $f_*$ is strictly positive in $(r,1)$. Consider the decomposition $f_* = f_1 + f_2$, where
\[f_1 = \chi_{[-r,r]}f_*\in \mathbb{V},\quad f_2 = (\chi_{[-1,-r]}+\chi_{[r,1]})f_*\in \mathbb{V}.\]
Since $r$ is a simple zero, $f_1$ and $f_2$ are not identically zero, and $L_{f_1} = l_{f_2} = r$. Then $f_1,f_2$ satisfy the assumption in Lemma \ref{lem:H1norm_decomposition}. Therefore, 
\[ \frac{\|f_*\|_{\dot{H}^{1/2}(\R)}^2}{\|f_*\|_{\dot{H}^1(\R)}^2} \leq \frac{\|f_1\|_{\dot{H}^{1/2}(\R)}^2}{\|f_1\|_{\dot{H}^1(\R)}^2} + \frac{\|f_2\|_{\dot{H}^{1/2}(\R)}^2}{\|f_2\|_{\dot{H}^1(\R)}^2}.\]
On the other hand, by the definition of $\mtx{S}$ and the optimality of $f_*$, we have that
\[\frac{\|f_1\|_{\dot{H}^{1/2}(\R)}^2}{\|f_1\|_{\dot{H}^1(\R)}^2} = r\ \frac{\|\mtx{S}(f_1)\|_{\dot{H}^{1/2}(\R)}^2}{\|\mtx{S}(f_1)\|_{\dot{H}^1(\R)}^2}\leq r\ \frac{\|f_*\|_{\dot{H}^{1/2}(\R)}^2}{\|f_*\|_{\dot{H}^1(\R)}^2},\]
and
\[\frac{\|f_2\|_{\dot{H}^{1/2}(\R)}^2}{\|f_2\|_{\dot{H}^1(\R)}^2}< \frac{\|\mtx{T}(f_2)\|_{\dot{H}^{1/2}(\R)}^2}{\|\mtx{T}(f_2)\|_{\dot{H}^1(\R)}^2} = (1-r)\ \frac{\|\mtx{S}(\mtx{T}(f_2))\|_{\dot{H}^{1/2}(\R)}^2}{\|\mtx{S}(\mtx{T}(f_2))\|_{\dot{H}^1(\R)}^2}\leq (1-r)\ \frac{\|f_*\|_{\dot{H}^{1/2}(\R)}^2}{\|f_*\|_{\dot{H}^1(\R)}^2}.\]
This further implies that 
\[\frac{\|f_1\|_{\dot{H}^{1/2}(\R)}^2}{\|f_1\|_{\dot{H}^1(\R)}^2} + \frac{\|f_2\|_{\dot{H}^{1/2}(\R)}^2}{\|f_2\|_{\dot{H}^1(\R)}^2} < \frac{\|f_*\|_{\dot{H}^{1/2}(\R)}^2}{\|f_*\|_{\dot{H}^1(\R)}^2},\]
which leads to a contradiction. Therefore, we conclude that $f_*$ must be strictly positive on $(0,1)$ up to a multiplicative constant.
\end{proof}

Having shown that $\om = -f_*$ is the unique solution (up to rescaling) to the relation \eqref{eqt:eigen_form} that is strictly negative on $(0,1)$, we can conclude that it coincides with the self-similar profile constructed in \cite{chen2021finite}.

\begin{corollary}\label{cor:equilibrium}
The self-similar profile $\Omega_*\in H^1(\R)$ of the De Gregorio model \eqref{eqt:DG} obtained in \cite{chen2021finite} is identical to the leading eigenfunction $f_*$ of $\mtx{M}$ up to rescaling.
\end{corollary} 

\begin{proof} The self-similar profile $\Omega_*\in H^1(\R)$ obtained in \cite{chen2021finite}, after proper rescaling, is odd in $x$, supported on $[-1,1]$, and strictly negative on $(0,1)$, and the corresponding $c_\om$ is given by $c_\om = (-\Delta)^{-1/2}\Omega_*(1)$. Therefore, by Proposition \ref{prop:necessary_condition}, $\Omega_*$ must be an eigenfunction of $\mtx{M}$ in $\mathbb{V}$. However, according to Corollary \ref{cor:sign-changing} and Theorem \ref{thm:uniqueness_positiveness}, the only eigenfunction of $\mtx{M}$ that is strictly negative on $(0,1)$ (up to a multiplicative constant) is the leading eigenfunction $f_*$. This proves the claimed result.
\end{proof}

Figure \ref{fig:profile_compare} in the appendix compares the leading eigenfunction $f_*$ of $\mtx{M}$ and the approximate self-similar profile $\widetilde{\Omega}_*$ numerically obtained in \cite{chen2021finite} by solving the dynamic rescaling equation \eqref{eqt:dynamic_rescaling}. We can see that the two profiles, when properly rescaled, are hardly distinguishable. Moreover, $\widetilde{\Omega}_*$ approximately solves the eigenvalue problem \eqref{eqt:eigen_problem} with the ratio $\mtx{M}(\widetilde{\Omega}_*)/\widetilde{\Omega}_*$ uniformly close to $\lambda_*$. 

\subsection{On the nonlinear stability}\label{sec:nonlinear_stability}
We discuss the similarities between the nonlinear stability results in \cite{chen2021finite} and in \cite{jia2019gregorio,lei2020constantin}. \\

In order to study the finite-time self-similar singularity of the De Gregorio model, Chen, Hou and Huang \cite{chen2021finite} applied the change of variables (modified form \cite{mclaughlin1986focusing,landman1988rate})
\[C_\om(t)\cdot\om\left(\frac{x}{C_l(t)},\tau(t)\right)\rightarrow \om(x,t),\quad C_\om(t)C_l(t)\cdot u\left(\frac{x}{C_l(t)},\tau(t)\right)\rightarrow u(x,t)\]
with 
\[C_\om(t) = \exp\left(\int_0^tc_\om(s)\idiff s\right),\quad C_l(t) = \exp\left(\int_0^tc_l(s)\idiff s\right),\quad \tau(t) =  \int_0^t C_\om(s)\idiff s\]
to reformulate the original equation \eqref{eqt:DG} into an equivalent form 
\begin{equation}\label{eqt:dynamic_rescaling}
\om_t + (c_l(t)x + u)\om_x = (c_\om(t) + u_x)\om,\quad u_x = \mtx{H}(\om),\quad u(0) = 0.
\end{equation}
Note that our notations are slightly different from those in \cite{chen2021finite}. We call \eqref{eqt:dynamic_rescaling} the dynamic rescaling formulation of \eqref{eqt:DG}. If a solution $(\om(x,t), c_\om(t),c_l(t))$ to the initial value problem of \eqref{eqt:dynamic_rescaling} converges to some equilibrium $(\Omega(x),c_\om,c_l)$ (which is a solution to the self-similar profile equation \eqref{eqt:main_equation}), then the corresponding solution to \eqref{eqt:DG} (recovered by the inverse change of variables) will develop a finite-time blowup asymptotically in the self-similar form \eqref{eqt:self-similar_solution}. Therefore, to prove finite-time self-similar singularity of the De Gregorio model \eqref{eqt:DG} from smooth initial data, one needs to prove the (local) convergence of \eqref{eqt:dynamic_rescaling} to some equilibrium $(\Omega(x),c_\om,c_l)$ such that $\Omega$ is non-trivial and $c_\om < 0$.

The time dependent functions $c_\om(t)$ and $c_l(t)$ are two degrees of freedom in the change of variables that determine the dynamic rescaling rates in the spatial scale and the magnitude. Recall that $\Omega_*\in \mathbb{V}$ denotes the self-similar profile constructed in \cite{chen2021finite} that is negative on $(0,1)$. Chen et al. showed that, under the constraint
\begin{equation}\label{eqt:normalization_condition}
c_\om(t) = c_l(t) = -u(1,t) = -\frac{1}{\pi}\int_{\R}\om(y,t)\ln|1-y|\idiff y,
\end{equation}
which is consistent with our condition \eqref{eqt:c_omega}, the equilibrium $\Omega_*$ of \eqref{eqt:dynamic_rescaling} is locally nonlinearly stable in the space $\mathbb{V}$ with respect to some weighted $H^1$-norm, hence proving the finite-time self-similar singularity of the De Gregorio model from some smooth, compactly supported data.\\

Next, we turn to the nonlinear stability established in \cite{jia2019gregorio,lei2020constantin}. Jia, Stewart and Sverak \cite{jia2019gregorio} proved that $A\sin(x-x_0)$ is a nonlinearly stable equilibrium of the De Gregorio model on $\mathbb{S}^1$, with $A$ and $x_0$ dependent on the initial data. Lei, Liu and Ren \cite{lei2020constantin} provided an alternative proof of this result. For the convenience of comparison, instead of considering the De Gregorio model on $\mathbb{S}^1$, we consider a modified De Gregorio model on $[-1,1]$:
\begin{equation}\label{eqt:modified_DG}
\om_t + u\om_x = u_x \om, \quad u = -(-\widetilde\Delta)^{-1/2}\om,\quad x\in [-1,1].
\end{equation}
Recall that $(-\widetilde\Delta)^{-1}$ is the inverse Laplacian on $[-1,1]$ associated with zero Dirichlet boundary conditions. It is not hard to check that, \eqref{eqt:modified_DG} is essentially equivalent to the De Gregorio model on $\mathbb{S}^1$ with a proper normalization condition on $u$. Then, the main result in \cite{jia2019gregorio,lei2020constantin} can translate to that $-\sin(\pi x)\in \mathbb{V}$ is a nonlinearly stable equilibrium of \eqref{eqt:modified_DG} given that the initial solution belongs to $\mathbb{V}$ and is properly normalized and sufficiently close to $-\sin(\pi x)$ in some weighted $H^1$-norm.\\

To see the similarity between these two nonlinear stability results, we first rewrite equation \eqref{eqt:dynamic_rescaling} under the constraint \eqref{eqt:normalization_condition} with initial data in $\mathbb{V}$ into a more compact form 
\begin{equation}\label{eqt:dynamic_rescaling_compact}
\om_t + \big[\mtx{M}(\om)\,,\,\om\big] = 0, \quad \om(x,0)\in \mathbb{V}.
\end{equation}
Here the commutator $[\cdot,\cdot]$ is defined as $[f,g] = f_x\,g - f\,g_x$. We can do this because a solution $\om(x,t)$ to \eqref{eqt:dynamic_rescaling} will remain in $\mathbb{V}$ for all $t\geq0$ if $\om(x,0)$ belongs to $\mathbb{V}$. The result in \cite{chen2021finite} and Corollary \ref{cor:equilibrium} together state that the leading eigenfunction $\Omega_*$ of $\mtx{M}$ is a nonlinearly stable equilibrium of \eqref{eqt:dynamic_rescaling_compact} with solutions restricted in $\mathbb{V}$.

Similarly, we can rewrite \eqref{eqt:modified_DG} with initial data in $\mathbb{V}$ as
\begin{equation}\label{eqt:modified_DG_compact}
\om_t + \big[(-\widetilde\Delta)^{-1/2}\om\,,\,\om\big] = 0, \quad \om(x,0)\in \mathbb{V}.
\end{equation}
The result in \cite{jia2019gregorio} implies that the leading eigenfunction $-\sin(\pi x)$ of $(-\widetilde\Delta)^{-1/2}$ is a nonlinearly stable equilibrium of \eqref{eqt:modified_DG_compact} with solutions properly normalized and restricted in $\mathbb{V}$.

It has been shown that $\mtx{M}$ and $(-\widetilde\Delta)^{-1/2}$ are both self-adjoint, positive definite compact operators on $\mathbb{V}$ with comparable sorted eigenvalues, and the semi-norms on $\mathbb{V}$ defined by them are equivalent. Moreover, their respective leading eigenfunctions, $f_*$ and $\sin(\pi x)$, are qualitatively similar (as shown in Figure \ref{fig:Eigenfunctions}(a)); they are both positive on $(0,1)$ and have a single bump. From this point of view, it is reasonable to believe that the nonlinear stability results established in \cite{chen2021finite} and in \cite{jia2019gregorio}, respectively, share the same nature. Interestingly, \eqref{eqt:dynamic_rescaling_compact} is the dynamic rescaling formulation of the De Gregorio model on the real line, while \eqref{eqt:modified_DG_compact} corresponds to the De Gregorio model on the circle.

More generally, we may consider an equation of the form
\begin{equation}\label{eqt:general_equation}
\om_t + \big[\mtx{A}(\om)\,,\,\om\big] = 0, \quad \om(x,0)\in \mathbb{V},
\end{equation}
where $\mtx{A}$ is some self-adjoint, compact operator on $\mathbb{V}$. Suppose that $\mtx{A}$ has a unique leading eigenfunction $f_*$, which is apparently an equilibrium of \eqref{eqt:general_equation}. Based on the the preceding discussions, it is reasonable to conjecture that, under some mild assumptions on $\mtx{A}$, any solution of \eqref{eqt:general_equation} that is properly normalized and initially close to $f_*$ will eventually converge to $f_*$ as $t\rightarrow+\infty$. This will be an interesting problem to study in future works. To get started, one may consider to establish the linear stability of \eqref{eqt:general_equation} around the equilibrium $f_*$ by generalizing the methods in \cite{jia2019gregorio,lei2020constantin} that deal with the linear stability of \eqref{eqt:modified_DG_compact}. If this conjecture is true, then we can give an alternative and probably simpler proof (that is purely analytic) of the finite-time self-similar singularity of the De Gregorio model on the real line from smooth initial data. Note that the only proof of this result so far by Chen et al. \cite{chen2021finite} relies heavily on computer assistance, which can be quite difficult to digest for most of the readers.

\subsection{An alternative formulation} 
We close this section with an alternative characterization of the eigenfunctions to the eigenvalue problem \eqref{eqt:eigen_problem}, which is of independent interest. As a byproduct, we can show that $f_*/x$ is strictly monotone on $(0,1]$.

Given a function $f\in \mathbb{V}\cap C^1([-1,1])$, we apply integration by parts to compute that 
\begin{align*}
(-\Delta)^{-1/2}f(x) &= -\frac{1}{\pi}\int_0^1 f(y)\ln\left|\frac{x-y}{x+y}\right| \idiff y \\
&= -\frac{1}{\pi}\int_0^1 \frac{f(y)}{y}\cdot \partial_y\left(\frac{y^2-x^2}{2}\ln\left|\frac{x-y}{x+y}\right|-xy\right)\idiff y \\
&= \frac{1}{\pi}\int_0^1 \left(\frac{f(y)}{y}\right)'\cdot \left(\frac{y^2-x^2}{2}\ln\left|\frac{x-y}{x+y}\right|-xy\right)\idiff y.
\end{align*}
We used $f(1) = 0$ and the odd symmetry of $f$. Hence, we can write
\begin{equation}\label{eqt:mid_step}
\frac{(-\Delta)^{-1/2}f(x)}{x} = \frac{1}{\pi}\int_0^1 \left(\frac{f(y)}{y}\right)'\, y\big(F(x/y)-1\big)\idiff y,
\end{equation}
where 
\begin{equation}\label{eqt:F_defition}
F(t):= \frac{1-t^2}{2t}\ln\left|\frac{t-1}{t+1}\right|, \quad t\geq 0.
\end{equation}

\begin{lemma}\label{lem:F_property} The function $F$ defined in \eqref{eqt:F_defition} satisfies
\begin{enumerate}
\item $F\in C([0,+\infty))$, $F(0) = -1$, $F(1) = 0$, $\lim_{t\rightarrow+\infty}F(t)=1$;
\item $F(1/t) = -F(t)$; 
\item $F'(t)>0$, $t\in (0,1)\cup (1,+\infty)$.
\end{enumerate}
\end{lemma}

\begin{proof} Properties (1) and (2) are straightforward to check. We only prove (3). For $t\in(0,1)$, $F(t)$ has the Taylor expansion 
\[F(t) = -\frac{1-t^2}{2t}\suml_{n=0}^\infty\frac{2t^{2n+1}}{2n+1} = -1 + \suml_{n=1}^\infty\frac{2t^{2n}}{4n^2-1}.\]
Hence, $F(t)$ is monotone increasing on $(0,1)$. By the symmetry property $F(t)= -F(1/t)$, we know $F(t)$  that is also monotone increasing on $(1,+\infty)$.
\end{proof}
 
We proceed by taking the derivative of \eqref{eqt:mid_step} to find that
\[\left(\frac{(-\Delta)^{-1/2}f(x)}{x}\right)' = \frac{1}{\pi}\int_0^1 \left(\frac{f(y)}{y}\right)'\,F'(x/y)\idiff y.\]
Multiplying both sides by $x$, we reach 
\begin{equation}\label{eqt:mid_step2}
x\left(\frac{(-\Delta)^{-1/2}f(x)}{x}\right)' = \frac{1}{\pi}\int_0^1 y\left(\frac{f(y)}{y}\right)'\,(x/y)F'(x/y)\idiff y.
\end{equation}
We can easily extend this formula to all functions in $\mathbb{V}$ by approximation theory.

Motivated by \eqref{eqt:mid_step2}, we define a linear operator $\mtx{N}$ acting on $L^2([0,1])$ as
\begin{equation}
\mtx{N}(g)(x) = \frac{1}{\pi}\int_0^1 g(y)K(x,y)\idiff y, \quad g\in L^2([0,1]).
\end{equation}
where 
\begin{equation}
K(x,y) = (x/y)\, F'(x/y)\geq 0.
\end{equation}
Clearly, $\mtx{N}$ maps $L^2([0,1])$ into itself. Note that, for any $x,y\geq 0$, it follows from property (2) of Lemma \ref{lem:F_property} that 
\[K(y,x) = (y/x)\, F'(y/x) = (y/x)\,(y/x)^{-2}\, F'(x/y)= (x/y)\, F'(x/y) = K(x,y).\]
That is, $K(x,y)$ is a symmetric kernel, implying that $\mtx{N}$ is self-adjoint with respect to $\la\cdot,\cdot\ra_{L^2}$. Moreover, $\mtx{N}$ is a compact operator, which follows from the fact that $K(\cdot\,,\cdot)\in L^2([0,1]^2)$ and thus $\mtx{N}$ is Hilbert--Schmidt. Nevertheless, we will prove the compactness of $\mtx{N}$ in an alternative way, as it demonstrates the relation between the eigen-pairs of $\mtx{N}$ and $\mtx{M}$.

We first prove an elementary result.

\begin{lemma}\label{lem:g_lemma}
For any $g\in L^2([0,1])$ and any $p>0$, 
\[\lim_{x\rightarrow0+}x^p\int_x^1\frac{g(y)^2}{y^p}\idiff y = 0.\] 
\end{lemma}

\begin{proof}
Since $g\in L^2([0,1])$, for any $\epsilon>0$, there exists some $x_1$ such that $\int_0^{x_1}g(y)^2\idiff y\leq\epsilon$. Let $x_2\in (0,x_1)$ be such that 
\[x_2^px_1^{-p}\cdot\|g\|_{L^2([0,1])}^2\leq \epsilon.\]
Then, for any $x\in(0,x_2)$, it holds that
\begin{align*}
x^p\int_x^1\frac{g(y)^2}{y^p}\idiff y &= x^p\int_x^{x_1}\frac{g(y)^2}{y^p}\idiff y + x^p\int_{x_1}^1\frac{g(y)^2}{y^p}\idiff y\\
&\leq  \int_x^{x_1}g(y)^2\idiff y + \frac{x^p}{x_1^p}\int_{x_1}^1g(y)^2\idiff y \\
&\leq 2\epsilon.
\end{align*}
This proves the lemma.
\end{proof}

We can prove the compactness of $\mtx{N}$ by relating it to the compactness of $\mtx{M}$.

\begin{lemma}\label{lem:compact_N}
The linear map $\mtx{N}:L^2([0,1])\to L^2([0,1])$ is compact. In particular, 
\[\|\mtx{N}(g)\|_{\dot{H}^1([0,1])}\leq \|g\|_{L^2([0,1])}, \quad g\in L^2([0,1]).\]
\end{lemma}

\begin{proof} Given any $g\in L^2([0,1])$, define a function $f$ as
\begin{equation}\label{eqt:f_from_g}
f(x) = -x\,\int_x^1\frac{g(y)}{y}\idiff y,\ x\in[0,1];\quad f(-x) = -f(x),\ x\in[-1,0);\quad f(x) = 0,\ x\in \R\backslash[-1,1].
\end{equation}
Note that $f$ is odd and supported $[-1,1]$, and $g(x) = x\partial_x(f(x)/x) = f_x(x)-f(x)/x$ for $x\in(0,1)$. Moreover, we have
\begin{align*}
\frac{f(x)^2}{x} &= x\left(\int_x^1\frac{g(y)}{y}\idiff y\right)^2\leq x\left(\int_x^1\frac{1}{y^{3/2}}\idiff y\right)\left(\int_x^1\frac{g(y)^2}{y^{1/2}}\idiff y\right)= 2x^{1/2}\left(\int_x^1\frac{g(y)^2}{y^{1/2}}\idiff y\right).
\end{align*}
It follows from Lemma \ref{lem:g_lemma} that $\lim_{x\rightarrow 0}f(x)^2/x = 0$. Hence, we can use integration by parts to obtain
\begin{equation}\label{eqt:mid_step3}
\|g\|_{L^2([0,1])}^2 = \int_0^1\left(f_x(x)^2 + \left(\frac{f(x)}{x}\right)^2 - 2\frac{f_x(x)f(x)}{x}\right)\idiff x = \int_0^1f_x(x)^2\idiff x = \frac{1}{2}\|f\|_{\dot{H}^1}^2.
\end{equation}
This means that $f\in \dot{H}^1([-1,1])$ and thus $f\in \mathbb{V}$.
By \eqref{eqt:mid_step2}, $\mtx{N}(g)$ can be expressed in terms of $\mtx{M}(f)$ as 
\begin{equation}\label{eqt:mid_step4}
\mtx{N}(g)(x) =  \frac{1}{\pi}\int_0^1 y\left(\frac{f(y)}{y}\right)'\,K(x,y)\idiff y = x\partial_x\left(\frac{(-\Delta)^{-1/2}f(x)}{x}\right) = x\partial_x\left(\frac{\mtx{M}(f)(x)}{x}\right),
\end{equation}
which implies that $\partial_x^2\mtx{M}(f)(x) = \partial_x\mtx{N}(g)(x) + \mtx{N}(g)(x)/x$ for $x\in(0,1)$. Also, it is easy to check that $\lim_{x\rightarrow 0}\mtx{N}(g)(x)^2/x = 0$ since $\mtx{M}(f)\in \dot{H}^2([-1,1])$ (see Theorem \ref{lem:compactness}). Hence, we can again use integration by parts to obtain 
\begin{align*}
\frac{1}{2}\|\mtx{M}(f)\|_{\dot{H}^2([-1,1])}^2 &= \int_0^1\left((\partial_x\mtx{N}(g)(x))^2 + \left(\frac{\mtx{N}(g)(x)}{x}\right)^2 + 2\frac{\mtx{N}(g)(x)\partial_x\mtx{N}(g)(x)}{x}\right)\idiff x\\
&=\int_0^1\left((\partial_x\mtx{N}(g)(x))^2 + 2\left(\frac{\mtx{N}(g)(x)}{x}\right)^2\right)\idiff x + (\mtx{N}(g)(1))^2\\
&\geq \|\mtx{N}(g)\|_{\dot{H}^1([0,1])}^2.
\end{align*}
Combining the estimates above and using Theorem \ref{lem:compactness} yields 
\[\|\mtx{N}(g)\|_{\dot{H}^1([0,1])}^2\leq \frac{1}{2}\|\mtx{M}(f)\|_{\dot{H}^2([-1,1])}^2 \leq \frac{1}{2}\|f\|_{\dot{H}^1}^2 = \|g\|_{L^2([0,1])}^2,\]
which is the desired result.
\end{proof}

Now that we know $\mtx{N}$ is self-adjoint and compact, we may consider its eigenvalue problem
\begin{equation}\label{eqt:eigen_problem_2}
\lambda g = \mtx{N}(g),\quad \lambda\in \R,\ g\in L^2([0,1]),
\end{equation}
which admits countably infinite eigenpairs. 

Let $(\lambda,f)$ be an eigen-pair of the eigenvalue problem \eqref{eqt:eigen_problem}, and let $g = x\partial_x(f/x) = f_x-f/x$, $x\in[0,1]$. Since $f\in H^1(\R)$, it is easy to check that $\lim_{x\rightarrow 0}f(x)^2/x = 0$, and thus it follows from \eqref{eqt:mid_step3} that $\|g\|_{L^2([0,1])} = \|f\|_{\dot{H}^1}/\sqrt{2}<+\infty$. Moreover, by the calculation \eqref{eqt:mid_step4}, 
\[\mtx{N}(g)(x) = x\partial_x\left(\frac{\mtx{M}(f)(x)}{x}\right) = \lambda\, x\partial_x\left(\frac{f(x)}{x}\right) = \lambda g(x),\]
meaning that $(\lambda,g)$ is an eigen-pair of eigenvalue problem \eqref{eqt:eigen_problem_2}. 

Conversely, suppose $(\lambda, g)$ solves the eigenvalue problem \eqref{eqt:eigen_problem_2}, and define $f$ as in \eqref{eqt:f_from_g}. From the proof of Lemma \ref{lem:compact_N}, we can find that $f\in \mathbb{V}$, and $(\lambda,f)$ is an eigen-pair of \eqref{eqt:eigen_problem}.

Therefore, there is a one-to-one correspondence between the eigen-pairs of $\mtx{M}$ and those of $\mtx{N}$, which share the same eigenvalues. In particular, if $f_*$ is the leading eigenfunction of $\mtx{M}$ corresponding to the largest eigenvalue $\lambda_*$, then $g_* = x\partial_x(f_*(x)/x)$ must be the leading eigenfunction of $\mtx{N}$ that solves the variational problem associated with \eqref{eqt:eigen_problem_2}:
\begin{equation}\label{eqt:variational_form_P}
\sup_{g\in L^2([0,1])\backslash\{0\}} \frac{\langle g, \mtx{N}(g)\rangle_{L^2([0,1])}}{\|g\|_{L^2([0,1])}^2}.
\end{equation}
Moreover, since the kernel $K(x,y)$ is non-negative, we have
\[\langle g_*, \mtx{N}(g_*)\rangle_{L^2([0,1])} = \frac{1}{\pi}\int_0^1\int_0^1g_*(x)g_*(y)K(x,y)\idiff x\idiff y\leq \frac{1}{\pi}\int_0^1\int_0^1|g_*(x)||g_*(y)|K(x,y)\idiff x\idiff y.\]
By the optimality of $g_*$, $g_*(x)$ cannot change sign on $(0,1)$. Without loss of generality, suppose that $g_* = |g_*|$. Since $K(x,y)$ is strictly positive for $x,y>0$, we have 
\[\lambda_* g_*(x) = \mtx{N}(g_*)(x) =  \frac{1}{\pi}\int_0^1 |g(y)| K(x,y)\idiff y >0,\quad \forall x\in(0,1].\]
Hence, $g_*(x)$ is strictly positive on $(0,1]$ up to a multiplicative constant. This implies that $f_*/x$ is strictly monotone on $(0,1]$. Since $f_*(1)=0$, this again proves that $f_*$ is strictly negative on $(0,1)$ up to a multiplicative constant.

\section{The general class}\label{sec:general_class}
This section aims at proving the second part of Theorem \ref{thm:main_theorem}: The existence of uncountably infinite solutions in the general class \eqref{general_claiss_formal}. As a recap, our goal is to construct more complicated solutions to equation \eqref{eqt:main_equation} in $\mathbb{V}$ that satisfy the relation
\begin{equation}\label{eqt:piecewise_eigen}
\mu_i\, \om(x) = \mtx{M}(\om)(x), \quad x\in [-x_i,-x_{i-1}]\cup[x_{i-1},x_i],\quad i=1,\dots,n,
\end{equation}
for some integer $n\geq 2$, some distinct points $0=x_0<x_1<\dots<x_{n-1}<x_n=1$, and some sequence $\{\mu_i\}_{i=1}^n$ of positive numbers such that $\mu_i\neq \mu_{i-1}$. Note that, since $\om,\mtx{M}(\om)\in H^1(\R)$, each $x_i$ must be a zero of $\om$, i.e. $\om(\pm x_i)=0, i=1,\dots,n$. We will also prove the desired regularity of these solutions.

\subsection{A nonlinear eigenvalue problem}
Denote
\begin{equation}\label{eqn: def of I_i}
I_i := [-x_i, -x_{i-1}]\cup [x_{i-1},x_i].
\end{equation}
The above equation \eqref{eqt:piecewise_eigen} can be equivalently stated as follows: there exists $\lam>0$ and $\{\d_i\}_{i=1}^n$ satisfying $\sum_{i = 1}^n \d_i = 0$, such that
\[
\suml_{i = 1}^n \chi_{I_i}(1+\d_i)\mtx{M}(\om) = \lam\om.
\]
This gives $\mu_i = \lambda/(1+\delta_i)$. Therefore, our task is to solve the nonlinear eigenvalue problem
\begin{equation}\label{eqn: nonlinear eigenvalue problem}
\mtx{M}(f) + \suml_{i = 1}^n \chi_{I_i} \d_i \mtx{M}(f) = \lam f,
\end{equation}
where $\{I_i\}_{i=1}^n$ are intervals defined by some of the zeros of $f$.

The idea of constructing such a solution is to perturb a sign-changing eigenfunction of $\mtx{M}$.
Take an eigenpair $(\lam_0,f_0)$ of $\mtx{M}$, such that $\lam_0 >0$, $\|f_0\|_{\dot{H}^1([-1,1])} = 1$, $f_0$ is sign-changing on $[0,1]$, and $f_0$ has at least $(n+1)$ distinct zeros on $[0,1]$.
According to our analysis in Section \ref{sec:basic_class}, such an eigen-pair must exist, and $\lam_0$ is a simple eigenvalue. Moreover, the number of zeros of $f_0$ can be arbitrarily large (see Corollary \ref{cor:sign-changing}).
Since the equation \eqref{eqt:piecewise_eigen} is formulated on $n$ different pieces, for simplicity, we shall assume that $f_0$ has exactly $(n+1)$ zeros on $[0,1]$, namely, $0=x_0<x_1<\cdots<x_n=1$.
Let $x_{-i} := -x_i$ $(i = 1,2,\cdots, n)$ denote its zeros on $[-1,0)$.
It will be clear that the general case of having more zeros than what we need can be discussed in a similar way.\\

We first prove that, if small $H^2$-perturbations are added to $f_0$, the number of zeros of the perturbed function should be the same as that of $f_0$, and the positions of these zeros are stable under the perturbation.

\begin{lemma}
\label{lem: stability of zeros}
Let $(\lam_0, f_0)$ be defined as above.
Denote $D_{r}:= \{h\in \mathbb{V}:\,\|h\|_{\dot{H}^2([-1,1])} \leq r\}$.
Then there exist $\epsilon >0$ and $r_0 \in (0,1]$ only depending on $f_0$, such that the followings hold:
\begin{enumerate}
  \item The open intervals $\{(x_i-\epsilon, x_i+\epsilon)\}_{i = -n}^n$ are non-overlapping;
  \item For any $h\in D_{r_0}$, $(\lam_0 f_0+h)$ has the same number of zeros on $[-1,1]$ as $\lam_0 f_0$, and all of them are simple zeros;
  \item If we denote the zeros of $(\lam_0 f_0+h)$ as $-1 = \tilde{x}_{-n}<\cdots <\tilde{x}_0 = 0 < \tilde{x}_1 <\cdots < \tilde{x}_n = 1$, then $\tilde{x}_i \in [-1,1]\cap (x_i-\epsilon, x_i+\epsilon)$ for all $i\in \{-n,\cdots, n\}$.
\end{enumerate}
Furthermore, for arbitrary $h_1,h_2\in D_{r_0}$, $(\lam_0 f_0+h_j)$, $j = 1,2$, both have exactly $(2n+1)$ zeros on $[-1,1]$, which we denote to be $\tilde{x}_{-n}^{(j)},\cdots, \tilde{x}_n^{(j)}$.
Then
\begin{equation*}
\big|\tilde{x}_i^{(1)}-\tilde{x}_i^{(2)}\big| \leq C\|h_1-h_2\|_{\dot{H}^1([-1,1])},\quad i \in \{-n,\cdots,n\},
\end{equation*}
where $C$ only depends on $f_0$.
\end{lemma}

\begin{proof}
It is known that
\begin{equation}\label{eqt: C-one-half estimate}
\|\lam_0 f_0\|_{\dot{C}^{1,1/2}([-1,1])} \leq C\|\mtx{M}(f_0)\|_{\dot{H}^2([-1,1])}\leq C\|f_0\|_{\dot{H}^1([-1,1])}=: C_0.
\end{equation}
On the other hand, by Theorem \ref{thm:non-degeneracy}, at every zero $\hat{x}$ of $f_0$ in $[-1,1]$,
\begin{equation}\label{eqt: lower bound for the derivative of lam_0f_0}
\lam_0 |\pa_x f_0(\hat{x})| \geq \frac{\lam_0^{3/2}}{2\pi^{1/2}}\|f_0\|_{\dot{H}^1} =: c_0.
\end{equation}
We also have
\[
\|h\|_{C^1([-1,1])}\leq C\|h\|_{\dot{H}^2([-1,1])}\leq Cr_0
\]
for any $h\in D_{r_0}$, where $C$ is a universal constant.

With $\epsilon = c_0^2/(16C_0^2)$, we define
\[
I := [-1,1]\cap\left(\cup_{i = -n}^n (x_i-\epsilon, x_i+\epsilon)\right).
\]
Clearly, $\inf_{[-1,1]\backslash I} |\lam_0 f_0|>0$, and it follows from \eqref{eqt: C-one-half estimate} and \eqref{eqt: lower bound for the derivative of lam_0f_0} that
$|\lam_0 \pa_x f_0|\geq c_0-C_0 \epsilon^{1/2} = \frac{3}{4} c_0$ on $I$. Then we may take $r_0$ to be suitably small, so that $\|\pa_x h\|_{L^\infty([-1,1])}\leq \frac{1}{4}c_0$ and  $\|h\|_{L^\infty([-1,1])}\leq \frac{1}{2}\inf_{[-1,1]\backslash I} |\lam_0 f_0|$.
Under these assumptions, we find that
\begin{enumerate}
  \item[(i)]\label{property: sign does not change on the other part of the interval} $(\lam_0 f_0 +h)$ has no zero in $[-1,1]\backslash I$, as it always has the same sign as $\lam_0 f_0$ there.
  \item[(ii)] On each sub-interval $[-1,1]\cap(x_i-\epsilon, x_i+\epsilon)$, $(\lam_0 f_0+h)$ is strictly monotone because $\pa_x (\lam_0 f_0+h)$ has the same sign as $\lam_0 \pa_x f_0(x_i)$.
      More precisely, $|\pa_x(\lam_0 f_0+h)|\geq \frac{1}{2} c_0$ on all these sub-intervals.
  \item[(iii)]\label{property: non-overlapping} These sub-intervals $(x_i-\epsilon,x_i+\epsilon)$, $i =  -n,\cdots, n$, are non-overlapping.
\end{enumerate}
These facts altogether imply that,
for arbitrary $h\in D_{r_0}$, $(\lam_0 f_0+h)$ has the same number of zeros on $[-1,1]$ as $\lam_0 f_0$, and $|\pa_x (\lam_0 f_0+h)(x)|\geq c_0/2$ in $I$. Moreover, $(\lam_0 f_0+h)$ has exactly one zero in each $(x_i-\epsilon,x_i+\epsilon)$.

Now take arbitrary $h_1,h_2\in D_{r_0}$.
We have shown that both $(\lam_0 f_0 + h_j)$, $j =1,2$, have exactly $(2n+1)$ zeros in $[-1,1]$, denoted by $\tilde{x}_{-n}^{(j)},\cdots, \tilde{x}_n^{(j)}$, and that $\tilde{x}_i^{(j)}$ locates in $[-1,1]\cap(x_i-\epsilon, x_i+\epsilon)$.
Since $|\pa_x(\lam_0 f_0+h_j)|\geq \frac{1}{2} c_0$ on each sub-interval $[-1,1]\cap(x_i-\epsilon, x_i+\epsilon)$, we derive that
\[
\big|\tilde{x}_i^{(1)}-\tilde{x}_i^{(2)}\big| \leq \frac{2}{c_0} \big\|(\lam_0f_0+h_1)-(\lam_0f_0+h_2)\big\|_{L^\infty} \leq C\|h_1-h_2\|_{\dot{H}^1([-1,1])},
\]
where $C$ depends on $f_0$.
\end{proof}

We can now establish the existence of uncountably infinite solutions in the general class \eqref{general_claiss_formal} via a perturbation argument. The claimed regularity properties of the general class will be verified in the next subsection.

\begin{theorem}
\label{thm:solution_general_class}
Let $(\lam_0,f_0)$ be defined as above.
Let $\d = (\d_1,\cdots, \d_n)$ denote an $n$-tuple such that $\sum_{i = 1}^n \d_i = 0$.
Define $|\delta|:=\max\{|\d_1|,\cdots, |\d_n|\}$.
There exists $r\in (0,1]$ and $\d_*>0$ depending on $\mtx{M}$ and $f_0$, such that the followings hold.
There exists a continuous family of functions $\{f_\d\}_{\{|\d|<\d_*\}} \subset \mathbb{V}$, satisfying that:
\begin{enumerate}
\item For each $\d$ with $|\d|<\d_*$, $f_\d$ is the unique solution to the nonlinear eigenvalue problem \eqref{eqn: nonlinear eigenvalue problem} such that $\langle f_\d - f_0,f_0\rangle_{\dot{H}^1} = 0$ and $\|f_\d - f_0\|_{\dot{H}^1}\leq r$.

\item $f_\d$ has exactly $(2n+1)$ distinct zeros on $[-1,1]$, which defines $\{I_i\}_{i=1}^n$ in \eqref{eqn: nonlinear eigenvalue problem} according to \eqref{eqn: def of I_i}, and
\begin{equation*}
\lam = \lam_0 + \left\langle \sum_{i = 1}^n \chi_{I_i} \d_i \mtx{M}(f_\d),\, f_0 \right\rangle_{\dot{H}^1}.
\end{equation*}

\item In addition, $f_\d$ has Lipschitz dependence on $\d$ in the $L^2([-1,1])$-topology, with $f_{(0,\cdots,0)} = f_0$.
To be precise, for any two $n$-tuples $\d,\d'$ such that $|\d|, |\d'|<\d_*$,
\[
\|f_\d - f_{\d'}\|_{L^2([-1,1])} \leq C\max_{i\in \{1,\cdots,n\}}|\d_i-\d_i'|,
\]
where $C$ depends on $\mtx{M}$ and $f_0$.
\item $c(f_\delta)\neq 0$ (see the definition of $c(f)$ in \eqref{eqt:cf}).
\end{enumerate}
As a corollary, the self-similar profile equation \eqref{eqt:main_equation} has uncountably infinite solutions in $\mathbb{V}$ that are not eigenfunctions of $\mtx{M}$.
\end{theorem}

\begin{proof}
Denote $\mathbb{V}_0 = \mathrm{span}\{f_0\}$, which is the eigenspace of $\mtx{M}$  corresponding to the simple eigenvalue $\lam_0$.
Let $ \mathbb{V}_0^\perp$ be its orthogonal complement in $\mathbb{V}$.
Denote the orthogonal projection from $\mathbb{V}$ to $\mathbb{V}_0^\perp$ to be $\mtx{P}$.

Fix an $n$-tuple $\d = (\d_1,\d_2,\cdots, \d_n)$ such that $\sum_{i = 1}^n \d_i = 0$ and $|\d_i| \ll 1$, where the smallness condition will be clear later.
Assume that a solution of \eqref{eqn: nonlinear eigenvalue problem} can be written as $f = f_0+g$ for some $g\in \mathbb{V}_0^\perp$.
Then \eqref{eqn: nonlinear eigenvalue problem} becomes
\begin{equation}\label{eqn: nonlinear eigenvalue problem recast}
(\mtx{M}-\lam_0 I)g = (\lam-\lam_0) (f_0+g) - \suml_{i = 1}^n \chi_{I_i} \d_i \mtx{M}(f_0+g).
\end{equation}
This is solvable if and only if the right-hand side is orthogonal to $\mathbb{V}_0$.
Indeed, since $\mtx{M}$ is compact, $(\mtx{M}-\lam_0 I)^{-1}$ is well-defined as a bounded operator mapping $\mathbb{V}_0^\perp$ into itself.
This leads to an equation for $\lam$,
\[
\left\langle (\lam-\lam_0) (f_0+g) - \suml_{i = 1}^n \chi_{I_i} \d_i \mtx{M}(f_0+g),\,
f_0 \right\rangle_{\dot{H}^1} = 0,
\]
which gives
\begin{equation}\label{eqn: formula for lambda}
\lam = \lam_0 + \left\langle \sum_{i = 1}^n \chi_{I_i} \d_i \mtx{M}(f_0+g),\, f_0 \right\rangle_{\dot{H}^1}.
\end{equation}
Plugging this into \eqref{eqn: nonlinear eigenvalue problem recast} yields
\begin{equation}
(\mtx{M}-\lam_0 I)g =
\left\langle \sum_{i = 1}^n \chi_{I_i} \d_i \mtx{M}(f_0+g),\, f_0 \right\rangle_{\dot{H}^1} g - \mtx{P}\left(\sum_{i = 1}^n \chi_{I_i} \d_i \mtx{M}(f_0+g)\right).
\label{eqn: nonlinear eigenvalue problem recast 2}
\end{equation}

Let $B_r:= \{g\in \mathbb{V}_0^\perp :\,\|g\|_{\dot{H}^1([-1,1])} \leq r\}$ with $r\in (0,1]$ to be determined according to $f_0$.
For this moment, we first require $r\leq r_0$ so that $\mtx{M}(B_r)\subset D_{r_0}$ (see Lemma \ref{lem:compactness}), where $r_0$ and $D_{r_0}$ are defined in Lemma \ref{lem: stability of zeros}.
In what follows, we will understand $B_r$ as a non-empty complete subset of $L^2([-1,1])$.
In view of \eqref{eqn: nonlinear eigenvalue problem recast 2}, consider the following mapping defined on $B_r$,
\[
\mtx{K}:\, g\mapsto
\big(\mtx{M}-\lam_0 I \big)^{-1}
\big[\langle \mtx{R}(g),\, f_0 \rangle_{\dot{H}^1}\,g
- \mtx{P}\big(\mtx{R}(g) \big)\big],
\]
where the mapping $\mtx{R}$ is defined by
\begin{equation}\label{eqn: operator R}
\mtx{R}(g) := \suml_{i = 1}^n \chi_{I_i(g)} \d_i \mtx{M}(f_0+g).
\end{equation}
Here $I_i(g)$ is determined by the zeros of $\mtx{M}(f_0+g)$ as in \eqref{eqn: def of I_i}.
This is well-defined thanks to Lemma \ref{lem: stability of zeros}.
Indeed, since we assumed $\mtx{M}(g)\in D_{r_0}$, by Lemma \ref{lem: stability of zeros}, $\mtx{M}(f_0 +g)$ is an odd function having exactly $(2n+1)$ zeros on $[-1,1]$.
We denote them to be $\tilde{x}_{-n} ,\cdots, \tilde{x}_n $, and we know that $\tilde{x}_i\in [-1,1]\cap(x_i-\epsilon, x_i+\epsilon)$.
Then we can define $I_i(g) : = [-\tilde{x}_i, -\tilde{x}_{i-1}]\cup [\tilde{x}_{i-1},\tilde{x}_i]$.

For the operator $\mtx{R}$, we derive that
\begin{align*}
\| \mtx{R}(g)\|_{\dot{H}^1([-1,1])}
&= \left(\suml_{i = 1}^n \d_i^2 \big\| \mtx{M}(f_0+g)\big\|_{\dot{H}^1(I_i(g))}^2\right)^{1/2}\\
&\leq \max\{|\d_1|,\cdots, |\d_n|\} \big\| \mtx{M}(f_0+g)\big\|_{\dot{H}^1}
\leq  C|\d|,
\end{align*}
where $C$ only depends on $\mtx{M}$.
This further implies that 
\[
\|\mtx{K}(g)\|_{\dot{H}^1([-1,1])} \leq C|\d|.
\]
Hence, assuming $|\d|$ to be small, we can make $\mtx{K}$ map $B_r$ into itself.

Next we show that $\mtx{K}$ is a contraction mapping on $B_r$ in the $L^2([-1,1])$-topology.
We shall use an equivalent representation of $\mtx{K}$:
\[
\mtx{K}(g) = 
\big(\mtx{M}-\lam_0 I \big)^{-1}
\big[\langle \mtx{R}(g),\, f_0 \rangle_{\dot{H}^1} (f_0+ g)
- \mtx{R}(g)\big],
\]
Let $\mathbb{X}$ be the completion of $\mathbb{V}_0^\perp$ in the $L^2([-1,1])$-topology.
Thanks to Lemma \ref{lem:compactness}, $\mtx{M}$ can be extended to $\mathbb{X}$ as a continuous operator, still denoted by $\mtx{M}$.
Also by Lemma \ref{lem:compactness}, it is a compact operator on $\mathbb{X}$, with $\mtx{M}(\mathbb{X})\subset \mathbb{V}_0^\perp$.
Hence, its eigenfunctions in $\mathbb{X}$ are exactly those in $\mathbb{V}_0^\perp$, so $(\mtx{M}-\lam_0 I)$ does not have zero eigenvalue in $\mathbb{X}$.
By the Fredholm alternative and the open mapping theorem, $(\mtx{M}-\lam_0 I)^{-1}$ is a bounded operator in $\mathbb{X}$.
Therefore, for arbitrary $g_1,g_2\in B_r$,
\begin{equation}
\begin{split}
& \big\|\mtx{K}(g_1)-\mtx{K}(g_2)\big\|_{L^2([-1,1])}\\
&\leq C\big\|\langle \mtx{R}(g_1),\, f_0 \rangle_{\dot{H}^1} (f_0+ g_1)
-\langle \mtx{R}(g_2),\, f_0 \rangle_{\dot{H}^1} (f_0+g_2)
- \mtx{R}(g_1) + \mtx{R}(g_2) \big\|_{L^2([-1,1])}\\
&\leq C\big\|\langle \mtx{R}(g_1) - \mtx{R}(g_2),\, f_0 \rangle_{\dot{H}^1} (f_0+ g_1)
\big\|_{L^2([-1,1])}\\
&\quad + C\big\|\langle \mtx{R}(g_2),\, f_0 \rangle_{\dot{H}^1} (g_1-g_2)\big\|_{L^2([-1,1])}
+ C\big\|\mtx{R}(g_1)-\mtx{R}(g_2)\big\|_{L^2([-1,1])}\\
&= C\big|\langle \mtx{R}(g_1) - \mtx{R}(g_2),\, f_0 \rangle_{\dot{H}^1} \big| \|f_0+ g_1\|_{L^2([-1,1])}\\
&\quad + C\big|\langle \mtx{R}(g_2),\, f_0 \rangle_{\dot{H}^1}\big|\|g_1-g_2\|_{L^2([-1,1])}
+ C\big\|\mtx{R}(g_1)-\mtx{R}(g_2)\big\|_{L^2([-1,1])},
\end{split}
\label{eqn: difference of K(g_1) and K(g_2) in L^2}
\end{equation}
where $C$ depends on $\mtx{M}$ and $\lam_0$ (and thus essentially on $f_0$).
Since $\mtx{R}(g_j)\in H_0^1([-1,1])$, $j = 1,2$, we calculate by integration by parts that
\begin{align*}
\big|\langle \mtx{R}(g_1) - \mtx{R}(g_2),\, f_0 \rangle_{\dot{H}^1}\big|
&= \left|\int_{-1}^1 \big(\mtx{R}(g_1) - \mtx{R}(g_2)\big) \pa_x^2 f_0\,dx\right|\\
&\leq \|\mtx{R}(g_1) - \mtx{R}(g_2)\|_{L^2([-1,1])} \|f_0\|_{\dot{H}^2([-1,1])}\\
&\leq \lam_0^{-1}\|\mtx{R}(g_1) - \mtx{R}(g_2)\|_{L^2([-1,1])}.
\end{align*}
In the last step, we used Lemma \ref{lem:compactness} and the assumption $\|f_0\|_{\dot{H}^1} = 1$.
We may bound $|\langle \mtx{R}(g_2),\, f_0 \rangle_{\dot{H}^1}|$ similarly.
Hence,
\begin{align*}
& \big\|\mtx{K}(g_1)-\mtx{K}(g_2)\big\|_{L^2([-1,1])}\\
&\leq  C\|\mtx{R}(g_2)\|_{L^2([-1,1])}\|g_1-g_2\|_{L^2([-1,1])}
+ C\big\|\mtx{R}(g_1)-\mtx{R}(g_2)\big\|_{L^2([-1,1])},
\end{align*}
where $C$ depends on $\mtx{M}$ and $f_0$.
By Lemma \ref{lem: stability of the R operator} below, we obtain that
\[
\big\|\mtx{K}(g_1)-\mtx{K}(g_2)\big\|_{L^2([-1,1])}\\
\leq  C|\d| \|g_1-g_2\|_{L^2([-1,1])}.
\]
Therefore, as long as $|\d|$ is suitably small, $\mtx{K}$ is a contraction mapping from $B_r$ to itself in the $L^2([-1,1])$-metric.

Now by the contraction mapping theorem, $\mtx{K}$ has a unique fixed-point in $B_r$, denoted by $g_\d$, so
\[
\big(\mtx{M}-\lam_0 I \big) g_\d
=
\langle \mtx{R}(g_\d),\, f_0 \rangle_{\dot{H}^1} g_\d
- \mtx{P}\big(\mtx{R}(g_\d) \big),
\]
which is exactly \eqref{eqn: nonlinear eigenvalue problem recast 2}.
This further implies \eqref{eqn: nonlinear eigenvalue problem} holds with $f = f_\d := f_0+g_\d$ and with $\lam$ defined by \eqref{eqn: formula for lambda} where $g$ needs to be replaced by $g_\d$.
Note that by Lemma \ref{lem: stability of zeros}, $f_0+g_\d$ has the same number of zeros as $f_0$ on $[-1,1]$, and each $I_i$ in \eqref{eqn: nonlinear eigenvalue problem} is well-defined and non-empty.
It is clear that $g_{(0,\cdots,0)} = 0$ by the uniqueness.

Lastly, we show that $g_\d$ has Lipschitz dependence on $\d$ in the $L^2([-1,1])$-topology, which further implies Lipschitz dependence of $f_\d$ on $\d$.
To stress the $\d$-dependence, we shall write the operators $\mtx{K}$ and $\mtx{R}$ as $\mtx{K}_\d$ and $\mtx{R}_\d$ respectively in the rest of the proof.
Take two $n$-tuples $\d$ and $\d'$ such that the corresponding $g_\d$ and $g_{\d'}$ are well-defined in $B_r$.
We argue as in \eqref{eqn: difference of K(g_1) and K(g_2) in L^2}
\begin{align*}
&\|g_\d - g_{\d'}\|_{L^2} \\
&= \big\|\mtx{K}_\d(g_\d) - \mtx{K}_{\d'}(g_{\d'})\big\|_{L^2} \\
&\leq C\big\|\langle \mtx{R}_\d(g_\d),\, f_0 \rangle_{\dot{H}^1} (f_0+ g_\d)
- \mtx{R}_\d (g_\d)
- \langle \mtx{R}_{\d'}(g_{\d'}),\, f_0 \rangle_{\dot{H}^1} (f_0+g_{\d'})
+ \mtx{R}_{\d'}(g_{\d'}) \big\|_{L^2} \\
&\leq C\big\|\mtx{R}_\d (g_\d) - \mtx{R}_{\d'}(g_{\d'}) \big\|_{L^2}
+ C\big|\langle \mtx{R}_{\d'}(g_{\d'}),\, f_0 \rangle_{\dot{H}^1} \big|\,
\big\|g_\d - g_{\d'}\big\|_{L^2} \\
&\leq C\big\|\mtx{R}_\d (g_\d) - \mtx{R}_{\d}(g_{\d'}) \big\|_{L^2}
+ C\left\|\sum_{i = 1}^n \chi_{I_i(g_{\d'})} (\d_i-\d_i') \mtx{M}(f_0+g_{\d'}) \right\|_{L^2}\\
&\quad + C \|\mtx{R}_{\d'}(g_{\d'})\|_{L^2([-1,1])}
\big\|g_\d - g_{\d'}\big\|_{L^2} \\
&\leq C|\d| \|g_\d - g_{\d'}\|_{L^2}
+ C|\d-\d'|\| \mtx{M}(f_0+g_{\d'}) \|_{L^2}
+ C |\d'| \big\|g_\d - g_{\d'}\big\|_{L^2},
\end{align*}
where $C$ depends on $\mtx{M}$ and $f_0$.
In the last inequality, we used Lemma \ref{lem: stability of the R operator} below.
Here $|\d-\d'|: = \max_{i}|\d_i-\d_i'|$.
Assuming $|\d|$ and $|\d'|$ to be smaller if needed, we obtain that
\[
\|g_\d - g_{\d'}\|_{L^2}
\leq C|\d-\d'|\| \mtx{M}(f_0+g_{\d'}) \|_{L^2}
\leq C|\d-\d'|,
\]
where $C$ depends on $\mtx{M}$ and $f_0$. This also implies $c(f_\delta)\neq 0$ provided that $|\delta|$ is small enough. In fact, by the definition \eqref{eqt:cf} we have 
\[|c(f_\delta) - c(f_0)|\leq C\|f_\delta - f_0\|_{L^2} = C\|g_\delta\|_{L^2} \leq C|\delta|. \]
Since $c(f_0)\neq 0$ (see Theorem \ref{thm:cf_nonzero}), $c(f_\delta)$ is also non-zero when $|\delta|$ is sufficiently small.

This completes the proof.
\end{proof}

In the above proof, we have used the following estimate.

\begin{lemma}\label{lem: stability of the R operator}
Fix an $n$-tuple $\delta$ as before. Let $B_r$ be defined as in the proof of Theorem \ref{thm:solution_general_class}.
Let the operator $\mtx{R}$ be defined in \eqref{eqn: operator R}.
Then for arbitrary $g_1,g_2\in B_r$,
\begin{equation}
\big\|\mtx{R}(g_1)-\mtx{R}(g_2)\big\|_{L^2([-1,1])}
\leq C |\d| \|g_1-g_2\|_{L^2([-1,1])},
\end{equation}
where the constant $C$ depends on $\mtx{M}$ and $f_0$.
Here $|\d|=\max\{|\d_1|,\cdots, |\d_n|\}$.
\end{lemma}

\begin{proof}
Since $g_1,g_2\in B_r$, by Lemma \ref{lem: stability of zeros},  $\mtx{M}(f_0+g_j)$, $j = 1,2$, has exactly $(2n+1)$ roots on $[-1,1]$.
We denote them to be $-1=\tilde{x}_{-n}^{(j)} < \cdots< \tilde{x}_n^{(j)}=1$ as before.
Lemma \ref{lem: stability of zeros} implies that, for all $l\in \{-n,\cdots, n\}$,
\[
\big|\tilde{x}_{l}^{(1)}-\tilde{x}_{l}^{(2)}\big|
\leq
C \|\mtx{M}(g_1-g_2)\|_{\dot{H}^1([-1,1])}
\leq
C_* \|g_1-g_2\|_{L^2} =: \epsilon_*,
\]
where $C_*$ only depends on $f_0$.
Let $I_i(g_j)$, $j=1,2$, $i=1,\dots,n$, be determined by zeros of $g_j$ as in \eqref{eqn: def of I_i}. Denote $J_i := I_i(g_1)\cap I_i(g_2)$ and $J := \cup_{i = 1}^n J_i$.
Then for $j = 1,2$,
\[
[-1,1]\backslash J \subset \cup_{l = -n+1}^{n-1} \left[\tilde{x}_l^{(j)}-\epsilon_*, \, \tilde{x}_l^{(j)}+\epsilon_*\right].
\]
Then we derive that
\begin{align*}
&\quad\big\|\mtx{R}(g_1)-\mtx{R}(g_2)\big\|_{L^2([-1,1])}^2\\
&\leq \big\|\mtx{R}(g_1)-\mtx{R}(g_2)\big\|_{L^2(J)}^2 + C\|\mtx{R}(g_1)\|_{L^2([-1,1]\backslash J)}^2+C\|\mtx{R}(g_2)\|_{L^2([-1,1]\backslash J)}^2\\
&\leq \left\|\sum_{i = 1}^n \chi_{J_i} \d_i \mtx{M}(g_1-g_2)\right\|_{L^2(J)}^2 \\
&\quad + C|\d|^2 \sum_{l = -n+1}^{n-1} \left( \big\|\lam_0f_0 + \mtx{M}(g_1)\big\|_{L^2([\tilde{x}_{l}^{(1)}-\epsilon_*, \, \tilde{x}_{l}^{(1)}+\epsilon_*])}^2
+ \big\|\lam_0f_0 + \mtx{M}(g_2)\big\|_{L^2([\tilde{x}_{l}^{(2)}-\epsilon_*, \, \tilde{x}_{l}^{(2)}+\epsilon_*])}^2\right)\\
&\leq |\d|^2 \| \mtx{M}(g_1-g_2)\|_{L^2}^2 \\
&\quad + C|\d|^2 \epsilon_* \sum_{l = -n+1}^{n-1}\left(\big\|\lam_0f_0 + \mtx{M}(g_1)\big\|_{L^\infty([\tilde{x}_{l}^{(1)}-\epsilon_*, \, \tilde{x}_{l}^{(1)}+\epsilon_*])}^2
+ \big\|\lam_0f_0 + \mtx{M}(g_2)\big\|_{L^\infty([\tilde{x}_{l}^{(2)}-\epsilon_*, \, \tilde{x}_{l}^{(2)}+\epsilon_*])}^2\right)\\
&\leq  C |\d|^2 \|g_1-g_2\|_{L^2}^2
+ C|\d|^2 \epsilon_*^2\left(\big\|\lam_0f_0 + \mtx{M}(g_1)\big\|_{\dot{H}^1}^2
+ \big\|\lam_0f_0 + \mtx{M}(g_2)\big\|_{\dot{H}^1}^2\right)\\
&\leq C |\d|^2 \|g_1-g_2\|_{L^2}^2,
\end{align*}
where the constant $C$ depends on $\mtx{M}$ and $f_0$.
In the second last line, we used the facts that each $\tilde{x}_l^{(j)}$ is a zero of $\lam_0f_0 + \mtx{M}(g_j)$, and that $H^1_0([-1,1])\hookrightarrow C^{1/2}(\R)$.
Note that $n$ is a constant defined by $f_0$, so we can absorb it into the constant $C$.

This proves the desired inequality.
\end{proof}

We remark that the construction of solutions in Theorem \ref{thm:solution_general_class} for a particular integer $n$ (the number of pieces on $[0,1]$) also implies the existence of general class solutions corresponding to any integer $k$ in the range $2\leq k<n$. In fact, we can always choose $\delta$ so that $\delta_i\neq \delta_{i-1}$ for $i=2,\dots,k$, and $\delta_j=\delta_k$ for $j=k+1,\dots,n$.

Conceptually, our construction by perturbation only yields solutions to \eqref{eqt:piecewise_eigen} that are close to one of the eigenfunctions of $\mtx{M}$. It may be possible to construct more general solutions to \eqref{eqt:piecewise_eigen} with $\{\mu_i\}_{i=1}^n$ being more scattered, instead of concentrating around some eigenvalue of $\mtx{M}$.

\subsection{General properties} 
We will prove some general properties of the solutions in the general class \eqref{general_claiss_formal}, which completes the proof of the second part of Theorem \ref{thm:main_theorem}.\\

We first show the claimed regularity of the solutions in the general class.

\begin{theorem}\label{thm:regularity_general}
Let $(\mu_1,\cdots,\mu_n, f)\in \R_{>0}^n\times \mathbb{V}$ be a solution to \eqref{eqt:piecewise_eigen} for some $n\geq 2$ and some $0=x_0<x_1<\cdots<x_{n-1}<x_n=1$. Then, for each $i=1,\dots,n$, $f\in H^2([x_{i-1},x_i])$, and $f$ is smooth in the interior of $(x_{i-1},x_i)$.
\end{theorem}

\begin{proof}
By the derivative formula \eqref{eqt:linear_map_derivative}, we find that 
\[\|f\|_{\dot{H}^2([x_{i-1},x_i])} = \mu_i^{-1}\|\mtx{M}(f)\|_{\dot{H}^2([x_{i-1},x_i])} \leq \mu_i^{-1}\|\mtx{H}(f)\|_{\dot{H}^1([-1,1])}\leq \mu_i^{-1}\|f\|_{\dot{H}^1}.\]
This shows that $f\in H^2([x_{i-1},x_i])$ for each $i$.

Next, we prove that $f\in C^\infty([y_{i-1},y_i])$ for any $x_{i-1}<y_{i-1}<y_i<x_i$. The proof is in the same spirit as that of Theorem \ref{thm:eigenvalue_problem}. Denote $I_i = [-x_i,-x_{i-1}]\cup[x_{i-1},x_i]$. Again from the derivative formula $\eqref{eqt:linear_map_derivative}$ we find that, on $[y_{i-1},y_i]$, 
\begin{align*}
\partial_x^2f &=  -\mu_i^{-1} \mtx{H}(f_x)\\
&= \mu_i^{-1} \suml_{j=1}^n\mu_j^{-1}\mtx{H}\big(\chi_{I_j}(\mtx{H}(f) + c(f)\big)\\
&= \mu_i^{-2}\mtx{H}(\mtx{H}(f)) - \mu_i^{-2}\mtx{H}\big((1-\chi_{I_i})\mtx{H}(f)\big) + \mu_i^{-1}\suml_{j\neq i}\mu_j^{-1}\mtx{H}\big(\chi_{I_j}\mtx{H}(f)\big) + \mu_i^{-1}c(f)\suml_{j=1}^n\mu_j^{-1}\mtx{H}(\chi_{I_j})\\
&=: -\mu_i^{-2}f - \mu_i^{-2} g_f + \mu_i^{-1} \hat{g}_f + \mu_i^{-1}c(f) h_f,
\end{align*}
where 
\[g_f(x) := \mtx{H}\left((1-\chi_{I_i})\mtx{H}(f)\right)(x) = \frac{1}{\pi}\int_{\R\backslash I_i}\frac{\mtx{H}(f)(y)}{x-y}\idiff y,\]
\[\hat{g}_f(x) := \suml_{j\neq i}\mu_j^{-1}\mtx{H}\big(\chi_{I_j}\mtx{H}(f)\big)(x) = \suml_{j\neq i}\frac{1}{\mu_j\pi}\int_{I_j}\frac{\mtx{H}(f)(y)}{x-y}\idiff y,\]
and
\[h_f(x) := \suml_{j=1}^n\mu_j^{-1}\mtx{H}(\chi_{I_j})(x) = \frac{1}{\pi}\suml_{j=1}^n\mu_j^{-1}\left(\ln\left|\frac{x+x_j}{x+x_{j-1}}\right|+\ln\left|\frac{x-x_{j-1}}{x-x_j}\right|\right)\]
are all infinitely differentiable on $[y_{i-1},y_i]$. This inductively implies that 
\[\partial_x^{k+2}f = -\mu_i^{-2}\partial_x^kf - \mu_i^{-2}\partial_x^kg_f + \mu_i^{-1} \partial_x^k\hat{g}_f + \mu_i^{-1}c(f) \partial_x^kh_f \in C([y_{i-1},y_i])\]
for all $k\geq0$, and thus $f\in C^\infty([y_{i-1},y_i])$.
\end{proof}

The next theorem is an analog of Theorem \ref{thm:non-degeneracy}, providing estimates for the derivatives at the zeros of a solution in the general class \eqref{general_claiss_formal}.

\begin{theorem}\label{thm:non-degeneracy_general}
Let $(\mu_1,\cdots,\mu_n, f)\in \R_{>0}^n\times \mathbb{V}$ be a solution to \eqref{eqt:piecewise_eigen} for some $n\geq 2$ and some $0=x_0<x_1<\cdots<x_{n-1}<x_n=1$. Let $\mu_{\max} = \max_{i}\mu_i$ and $\mu_{\min} = \min_{i}\mu_i$. Define $I_i:= [-x_i,-x_{i-1}]\cup[x_{i-1},x_i], i=1,\dots,n$. Then, for each $i$ and for any $r\in I_i$ such that $f(r)=0$, 
\begin{equation}\label{eqt:non-degeneracy_general}
(\partial_xf(r))^2 = \frac{c(f)^2}{\mu_i^2} + \frac{1}{\pi\mu_i^2}\suml_{j=1}^n\mu_j\int_{I_j} \frac{f(x)^2}{(r-x)^2}\idiff x\geq \frac{c(f)^2}{\mu_{\max}^2} + \frac{\mu_{\min}^3}{4\pi\mu_{\max}^2}\|f\|_{\dot{H}^1}^2.
\end{equation}
Here $\partial_xf(r)$ denotes the one-sided derivatives $\partial_{x-}f(r)$ if $r=-x_{i-1},x_i$ and $\partial_{x+}f(r)$ if $r= -x_i,x_{i-1}$. 
\end{theorem}  

\begin{proof}
Define $g_i := \chi_{I_i}f$, so that $f=\sum_{i=1}^ng_i$ and $\mtx{M}(f)=\sum_{i=1}^n\mu_ig_i$. Since $x_0,x_1,\dots,x_n$ must be zeros of $f$, we have $g_i\in \mathbb{V}$ for each $i$. \eqref{eqt:piecewise_eigen} and $\eqref{eqt:linear_map_derivative}$ together imply that 
\[\suml_{i=1}^n\mu_i \partial_xg_i = \partial_x\mtx{M}(f) = -\chi_{[-1,1]}\big(\mtx{H}(f) + c(f)\big).\]
From this we can derive the following on $[-1,1]$:
\begin{align*}
\suml_{i=1}^n\mu_i\mtx{H}(\partial_x(g_i^2)) &= 2\suml_{i=1}^n\mu_i\mtx{H}(f\partial_xg_i)\\
&= -2\mtx{H}\big(f(\mtx{H}(f)+c(f))\big)\\
&= -\mtx{H}(f)^2 + f^2 - 2c(f)\mtx{H}(f) \\
&= -(\mtx{H}(f)+c(f))^2 + f^2 + c(f)^2 \\
&= -\sum_{i=1}^n\mu_i^2(\partial_xg_i)^2 + f^2 + c(f)^2\\
&= -\sum_{i=1}^n\mu_i^2\chi_{I_i}(\partial_xf)^2 + f^2 + c(f)^2.
\end{align*}
Here we have used Tricomi's identity that $\mtx{H}(g\mtx{H}(g)) = (\mtx{H}(g)^2 - g^2)/2$ for any $g\in L^2(\R)$. On the other hand, for any $r\in [-1,1]$ such that $f(r)=0$, we have
\begin{align*}
\suml_{i=1}^n\mu_i\mtx{H}(\partial_x(g_i^2))(r) &= \suml_{i=1}^n\mu_i(-\Delta)^{1/2}(g_i^2)(r) \\
&= \suml_{i=1}^n\frac{\mu_i}{\pi}\cdot \text{P.V.}\int_{\R}\frac{g_i(r)^2-g_i(x)^2}{(r-x)^2}\idiff x \\
&= -\frac{1}{\pi}\suml_{i=1}^n\mu_i\int_{\R}\frac{g_i(x)^2}{(r-x)^2}\idiff x\\
&= -\frac{1}{\pi}\suml_{i=1}^n\mu_i\int_{I_i}\frac{f(x)^2}{(r-x)^2}\idiff x
\end{align*}
Therefore, if $r$ is a zero of $f$ in $I_i$, we have
\begin{align*}
(\partial_xf(r))^2 - \frac{c(f)^2}{\mu_i^2} &= \frac{1}{\pi\mu_i^2}\suml_{j=1}^n\mu_j\int_{I_j} \frac{f(x)^2}{(r-x)^2}\idiff x \geq \frac{\mu_{\min}}{4\pi\mu_{\max}^2}\|f\|_{L^2}^2\\
&\geq \frac{\mu_{\min}}{4\pi\mu_{\max}^2} \|\mtx{M}(f)\|_{\dot{H}^1}^2 = \frac{\mu_{\min}}{4\pi\mu_{\max}^2} \suml_{i=1}^n\mu_j^2\|g_j\|_{\dot{H}^1}^2 \geq \frac{\mu_{\min}^3}{4\pi\mu_{\max}^2} \|f\|_{\dot{H}^1}^2,
\end{align*}
which is \eqref{eqt:non-degeneracy_general}. 
\end{proof}

Finally, we show that all solutions in the general class \eqref{general_claiss_formal} must be sign-changing on $(0,1)$.

\begin{corollary}\label{cor:sign-changing_general}
Let $(\mu_1,\cdots,\mu_n, f)\in \R_{>0}^n\times \mathbb{V}$ be a solution to \eqref{eqt:piecewise_eigen} for some $n\geq 2$ and some $0=x_0<x_1<\cdots<x_{n-1}<x_n=1$. Suppose that $\|f\|_{H^1}>0$. Then, $f$ is sign-changing on $(0,1)$. In particular, $\sgn(f)(x_i+)\cdot \sgn(f)(x_i-) =-1$ for each $1\leq i<n$.
\end{corollary}

\begin{proof} Let $I_i$ be defined as in Theorem \ref{thm:non-degeneracy_general}. For each $1\leq i<n$, since $x_i\in I_i\cap I_{i+1}$ is a zero of $f$, it follows from \eqref{eqt:non-degeneracy_general} that 
\[|\partial_x\mtx{M}(f)(x_i)| = \mu_i|\partial_{x-}f(x_i)|= \mu_{i+1}|\partial_{x+}f(x_i)| >0.\]
Since $\mtx{M}(f)\in H^2([-1,1])\subset C^1([-1,1])$, $\mtx{M}(f)$ must be sign-changing in a small neighborhood of $x_i$. By the piecewise relation \eqref{eqt:piecewise_eigen}, $f$ has the same sign as $\mtx{M}(f)$. This proves the claim.
\end{proof}

We have now completed the proof of Theorem \ref{thm:main_theorem}. 

\appendix 

\section{Plots of the eigenfunctions}\label{sec:numerical}
We present some figures to visualize some of the theoretical results in this paper. Approximations of the eigen-pairs of $\mtx{M}$ are obtained by numerically solving the eigenvalue problem \eqref{eqt:eigen_problem}. We use the codes from \cite{Matlabcode} to numerically compute $\mtx{M}(f)$ for $f\in \mathbb{V}$.   

\begin{figure}[!ht]
\centering
    \begin{subfigure}[b]{0.32\textwidth}
        \includegraphics[width=1\textwidth]{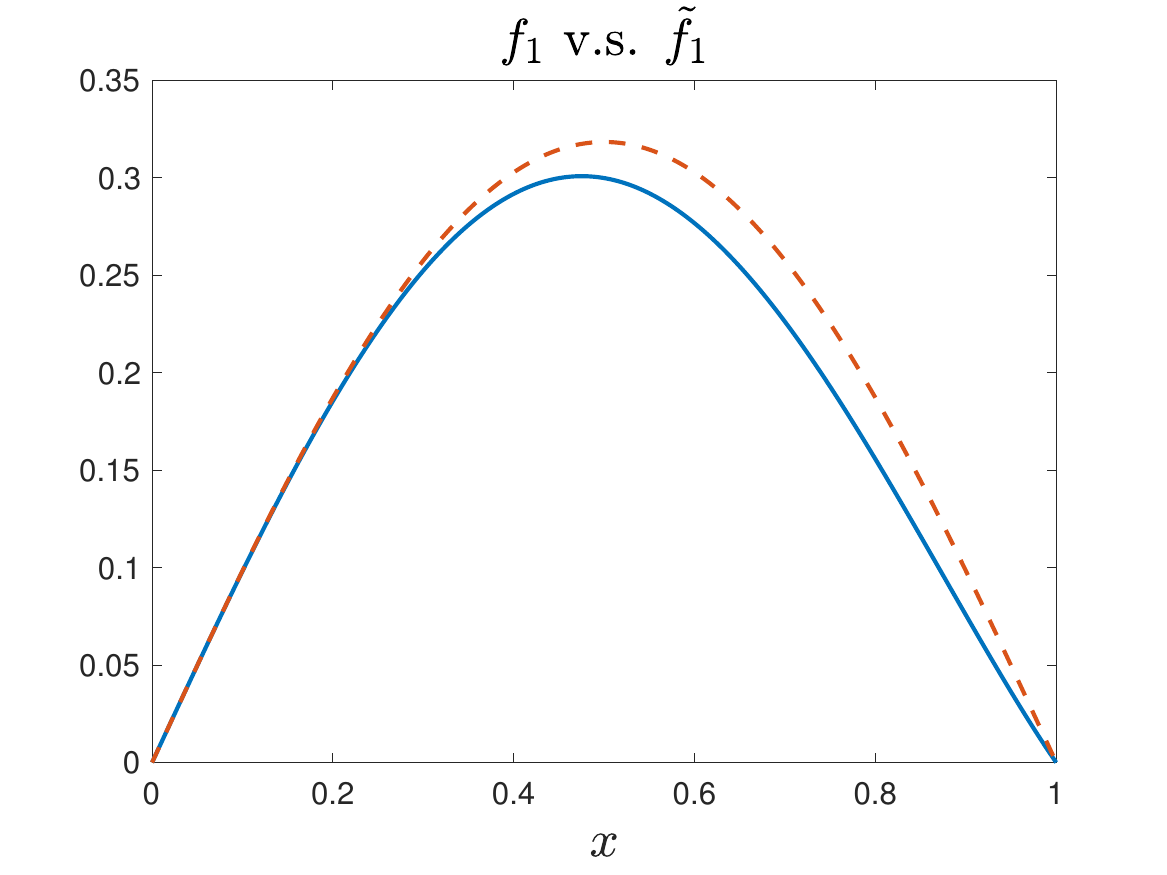}
        \caption{$\begin{array}{c}\lambda_1 \approx 0.2896 \\ 		\tilde{\lambda}_1 \approx 0.3183 \end{array}$}
    \end{subfigure}
    \begin{subfigure}[b]{0.32\textwidth}
        \includegraphics[width=1\textwidth]{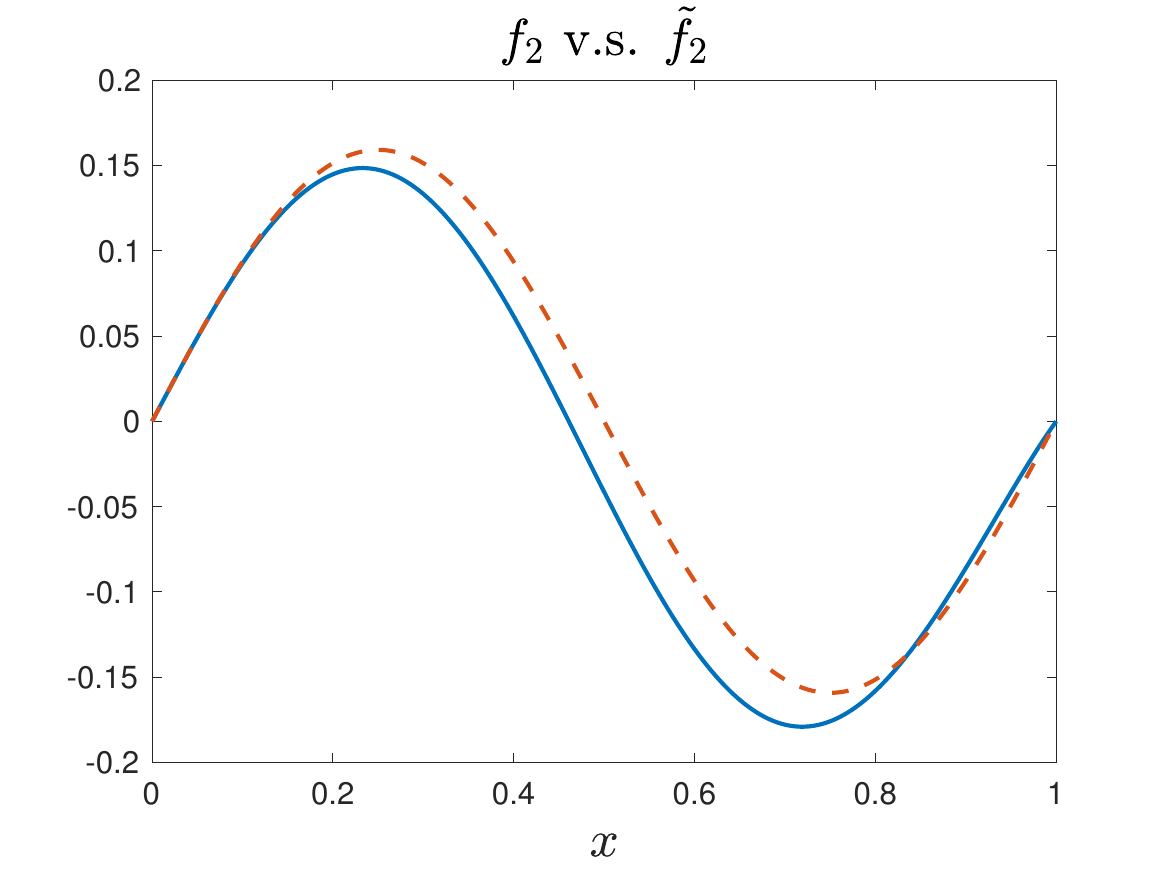}
        \caption{$\begin{array}{c}\lambda_2 \approx 0.1509 \\ 		\tilde{\lambda}_2 \approx 0.1592 \end{array}$}
    \end{subfigure}
    \begin{subfigure}[b]{0.32\textwidth}
        \includegraphics[width=1\textwidth]{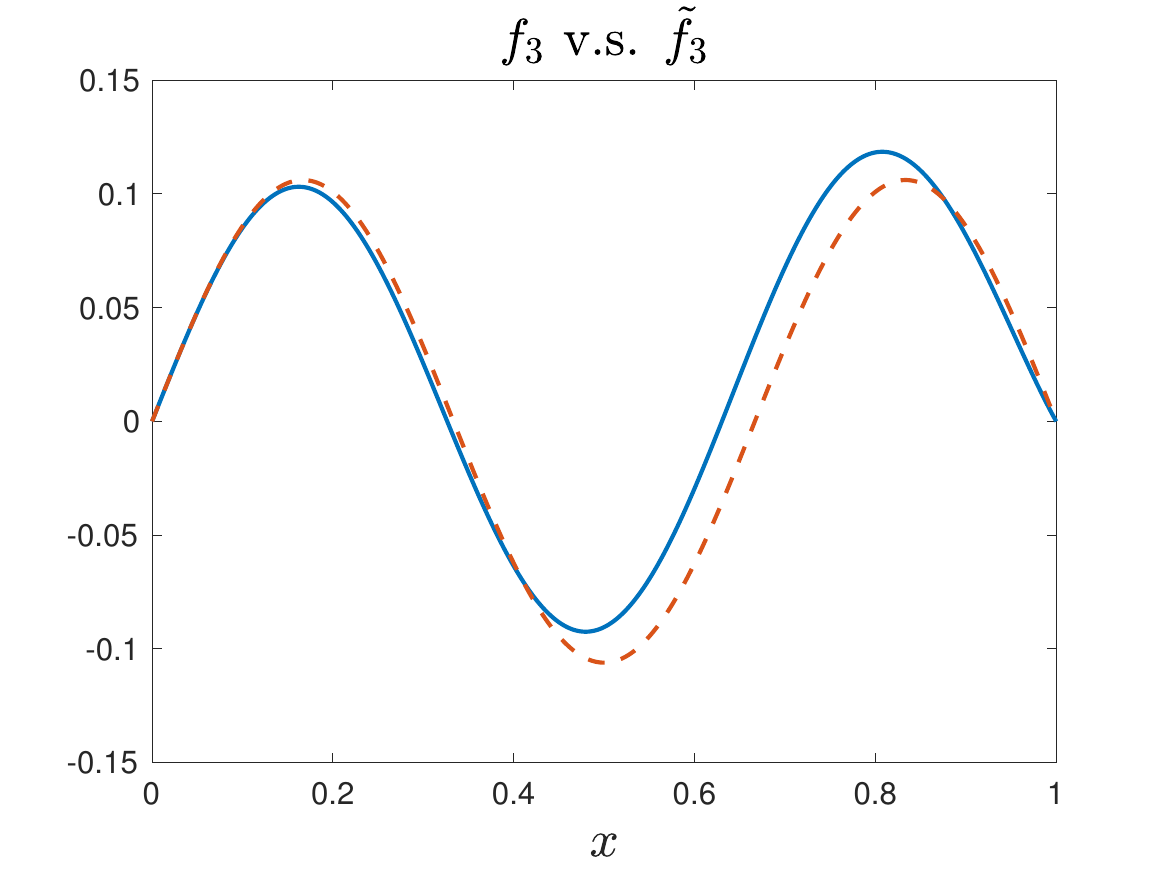}
        \caption{$\begin{array}{c}\lambda_3 \approx 0.1022 \\ 		\tilde{\lambda}_3 \approx 0.1061 \end{array}$}
    \end{subfigure}
    \begin{subfigure}[b]{0.32\textwidth}
        \includegraphics[width=1\textwidth]{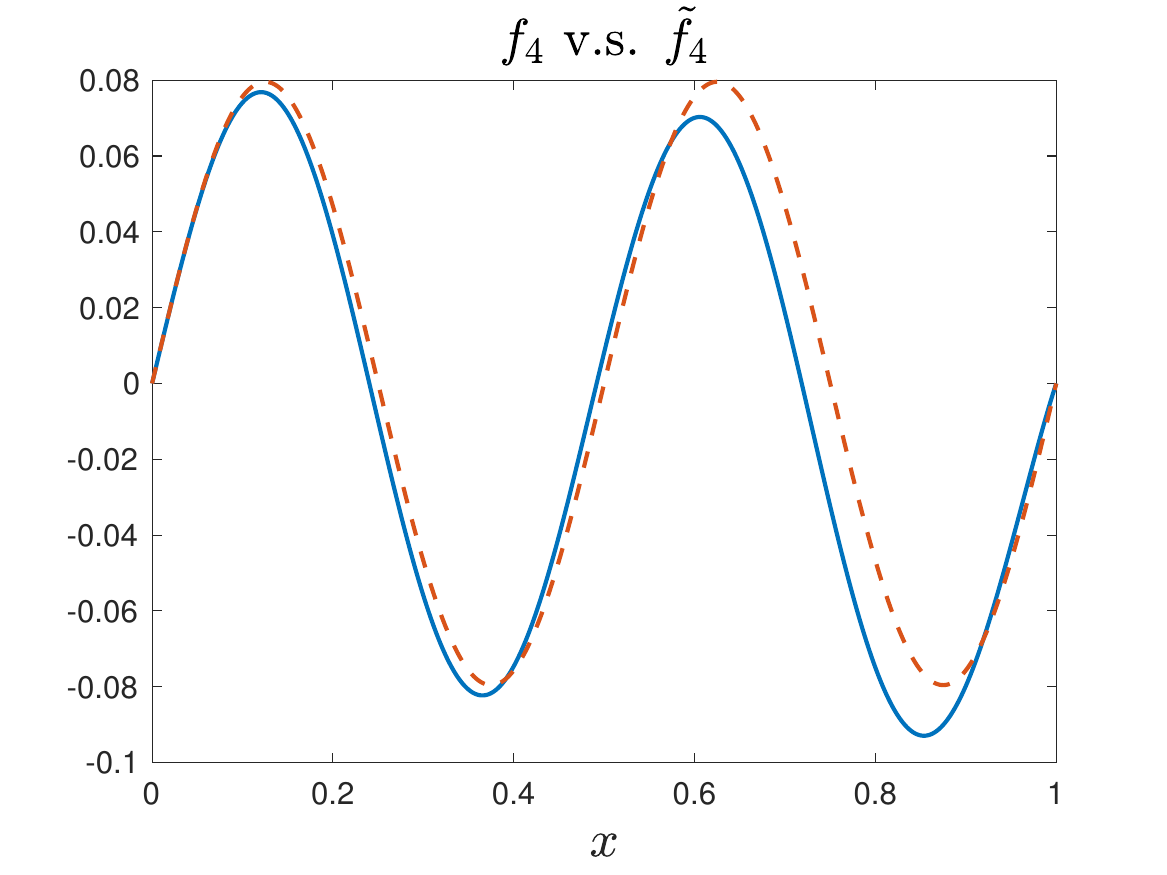}
        \caption{$\begin{array}{c}\lambda_4 \approx 0.0773 \\ 		\tilde{\lambda}_4 \approx 0.0796 \end{array}$}
    \end{subfigure}
    \begin{subfigure}[b]{0.32\textwidth}
        \includegraphics[width=1\textwidth]{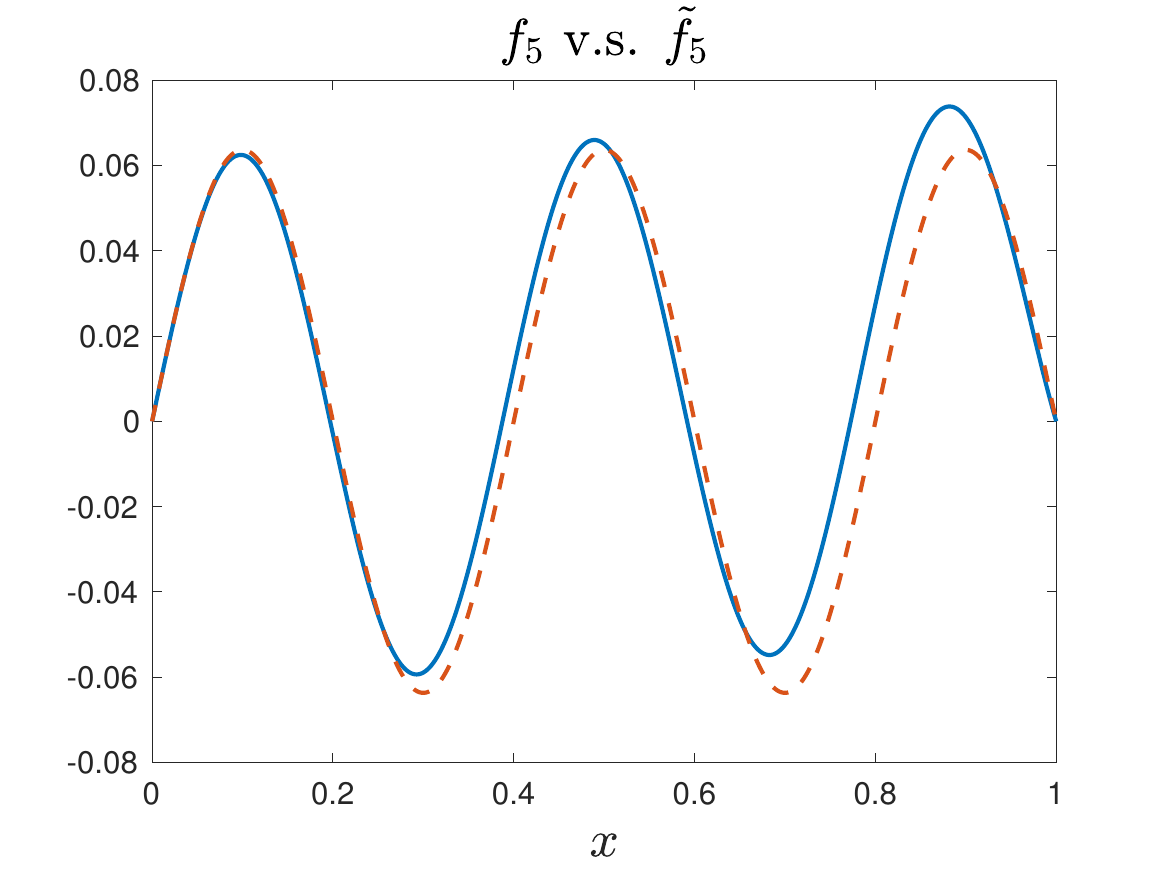}
        \caption{$\begin{array}{c}\lambda_5 \approx 0.0622 \\ 		\tilde{\lambda}_5 \approx 0.0637 \end{array}$}
    \end{subfigure}
    \begin{subfigure}[b]{0.32\textwidth}
        \includegraphics[width=1\textwidth]{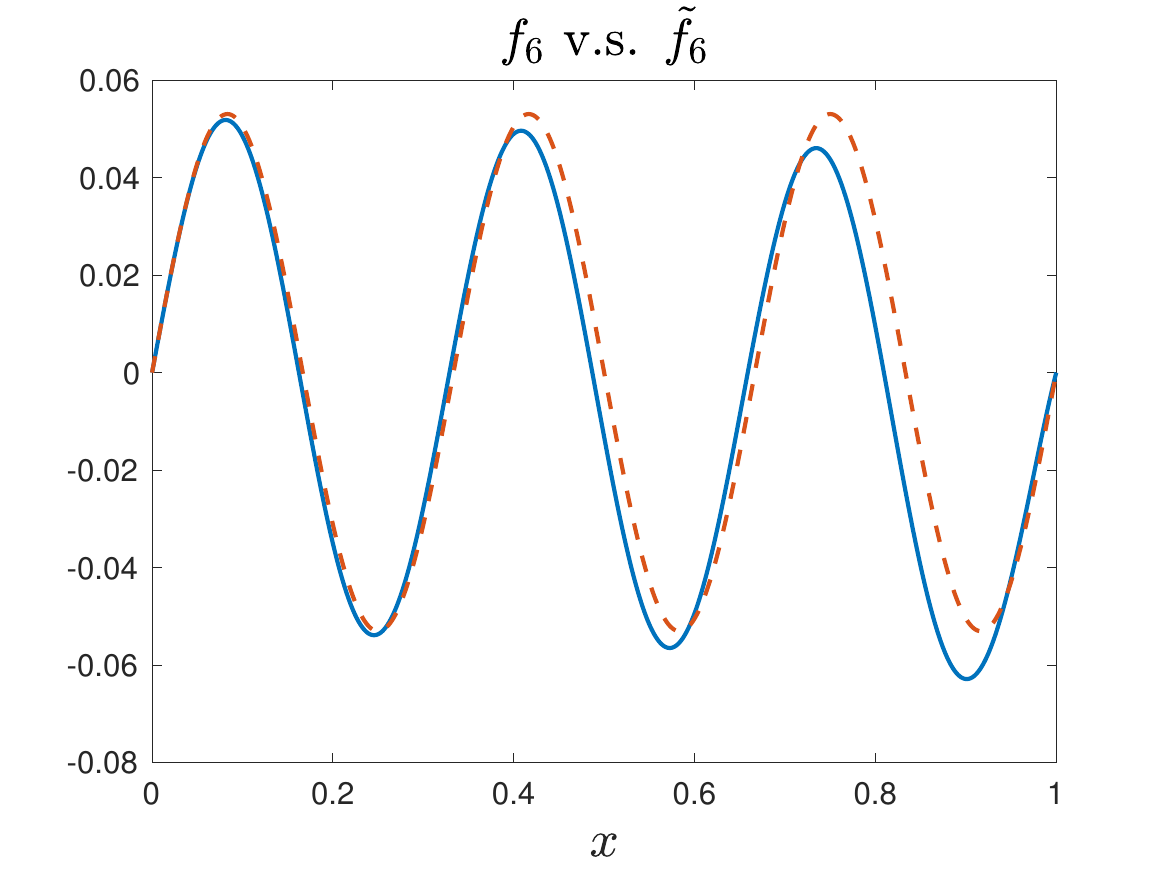}
        \caption{$\begin{array}{c}\lambda_6 \approx 0.0520 \\ 		\tilde{\lambda}_6 \approx 0.0531 \end{array}$}
    \end{subfigure}
    \caption[Eigenfunctions]{Plots of the eigenfunctions corresponding to the largest $6$ eigenvalues of $\mtx{M}$ (blue solid curves), versus the corresponding eigenfunctions of $(-\widetilde\Delta)^{-1/2}$ (red dashed curves). All plotted functions are normalized so that their derivatives at $x=0$ are $1$. Recall that $\lambda_n$ and $\tilde{\lambda}_n$ denote the eigenvalues of $\mtx{M}$ and $(-\widetilde\Delta)^{-1/2}$, respectively.}  
     \label{fig:Eigenfunctions}
\end{figure}

\begin{figure}[!ht]
\centering
    \begin{subfigure}[b]{0.32\textwidth}
        \includegraphics[width=1\textwidth]{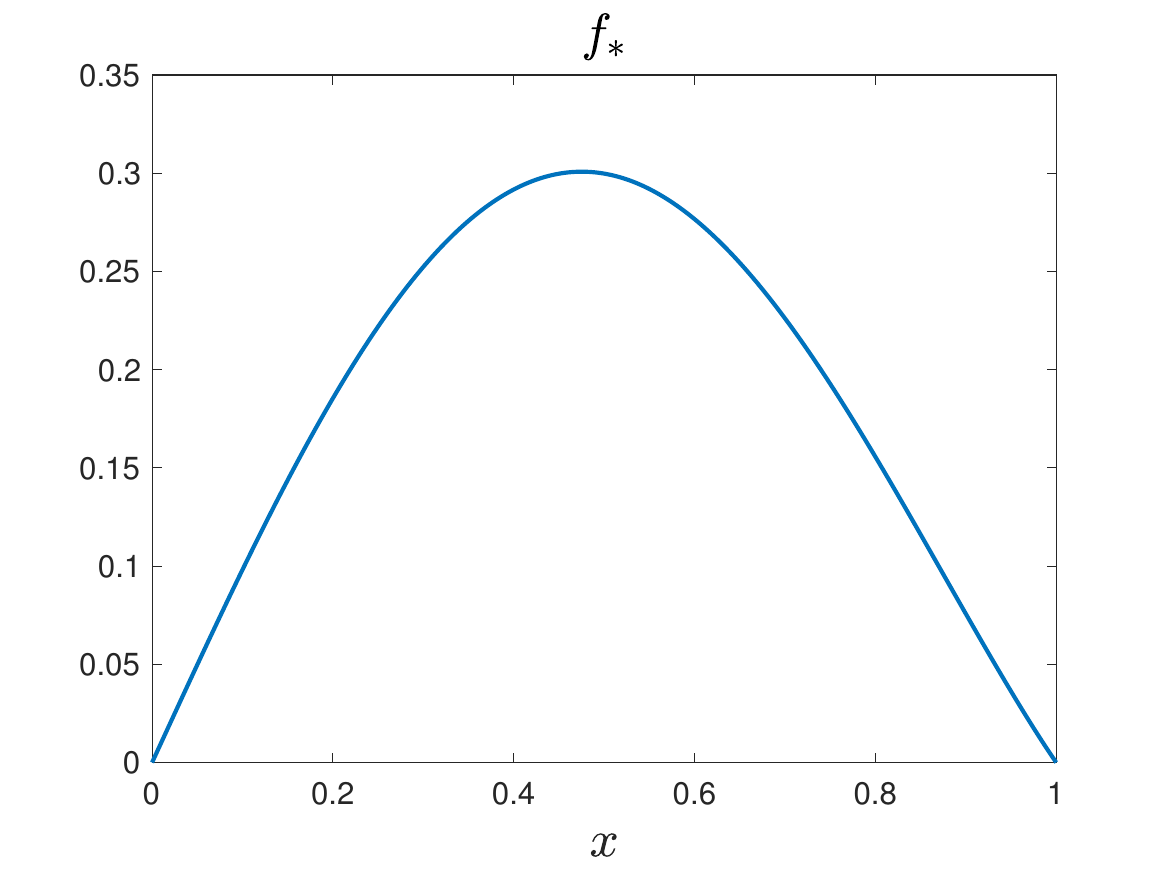}
        \caption{$f_*$}
    \end{subfigure}
    \begin{subfigure}[b]{0.32\textwidth}
        \includegraphics[width=1\textwidth]{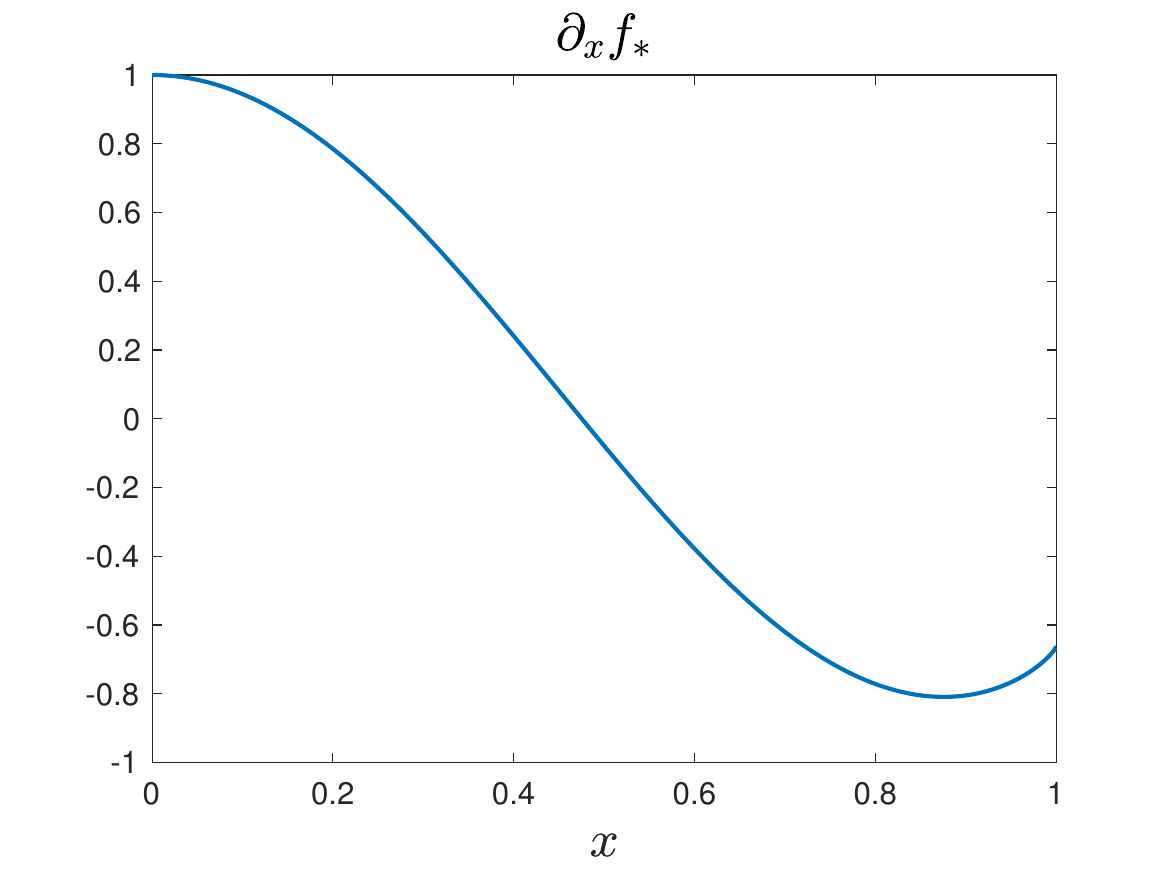}
        \caption{$\partial_xf_*$}
    \end{subfigure}
    \begin{subfigure}[b]{0.32\textwidth}
        \includegraphics[width=1\textwidth]{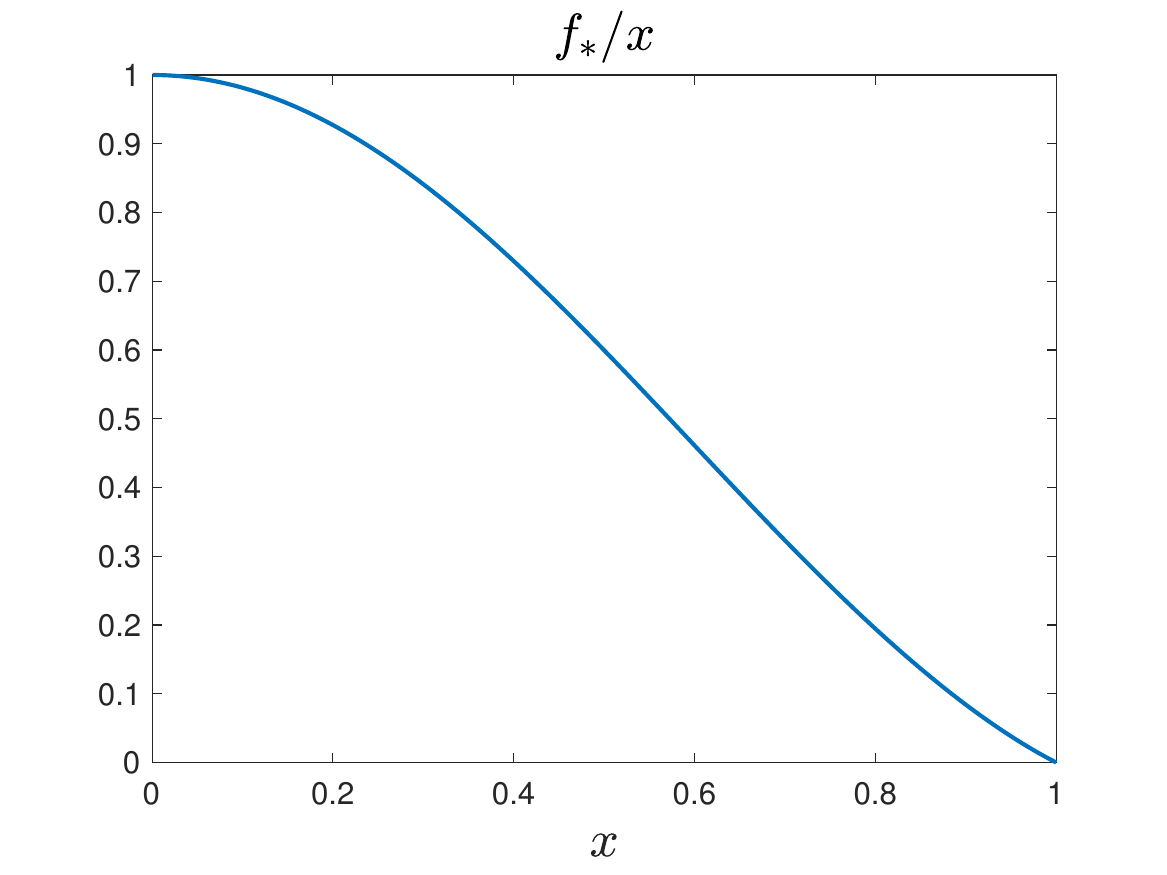}
        \caption{$f_*/x$}
    \end{subfigure}
    \caption[Leading Eigenfunction]{From left to right: $f_*$, $\partial_xf_*$, and $f_*/x$. $f_*$ denotes the leading eigenfuntion of $\mtx{M}$. We numerically verify that $f_*$ is strictly positive on $(0,1)$ and $f_*/x$ is strictly decreasing on $(0,1)$ up to a multiplicative constant.}  
     \label{fig:lead_eigen}
\end{figure}

\begin{figure}[!ht]
\centering
    \begin{subfigure}[b]{0.32\textwidth}
        \includegraphics[width=1\textwidth]{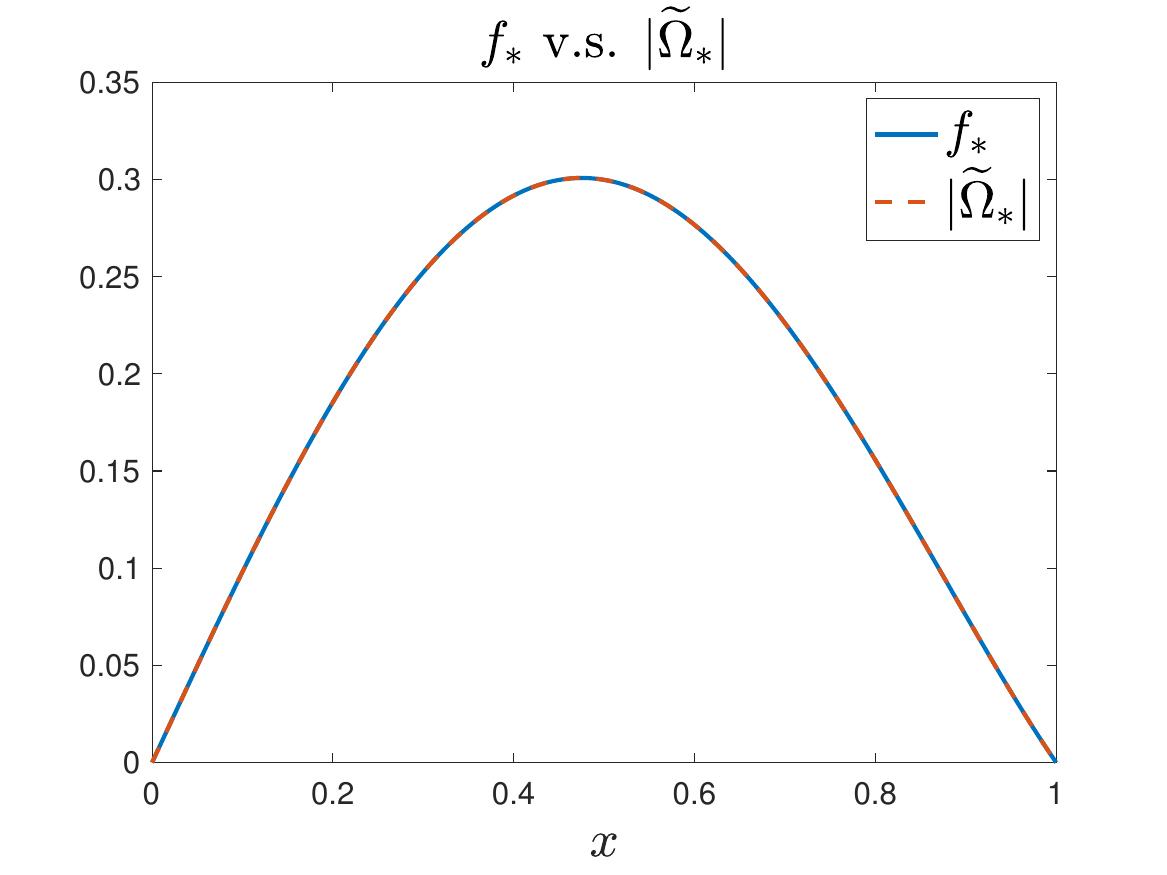}
        \caption{$f_*$ v.s. $|\widetilde{\Omega}_*|$}
    \end{subfigure}
    \begin{subfigure}[b]{0.32\textwidth}
        \includegraphics[width=1\textwidth]{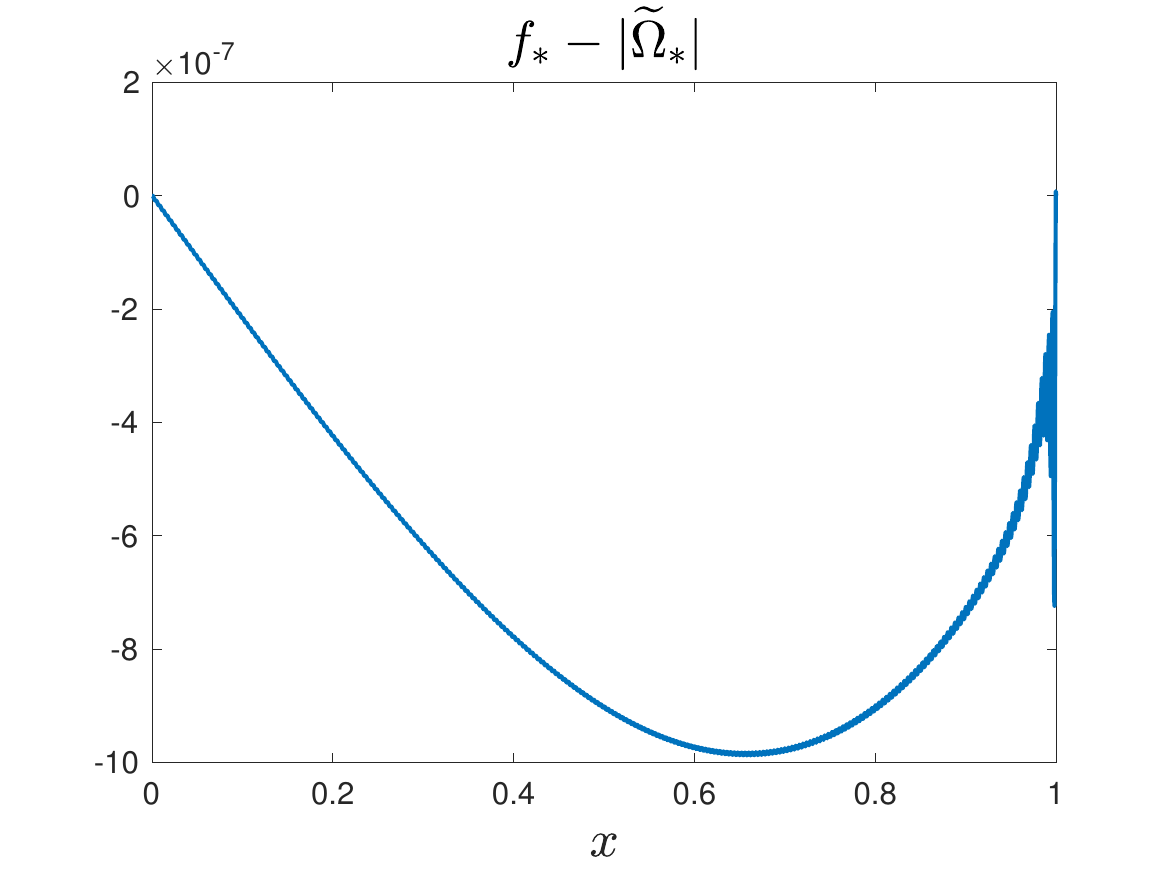}
        \caption{$f_* - |\widetilde{\Omega}_*|$}
    \end{subfigure}
    \begin{subfigure}[b]{0.32\textwidth}
        \includegraphics[width=1\textwidth]{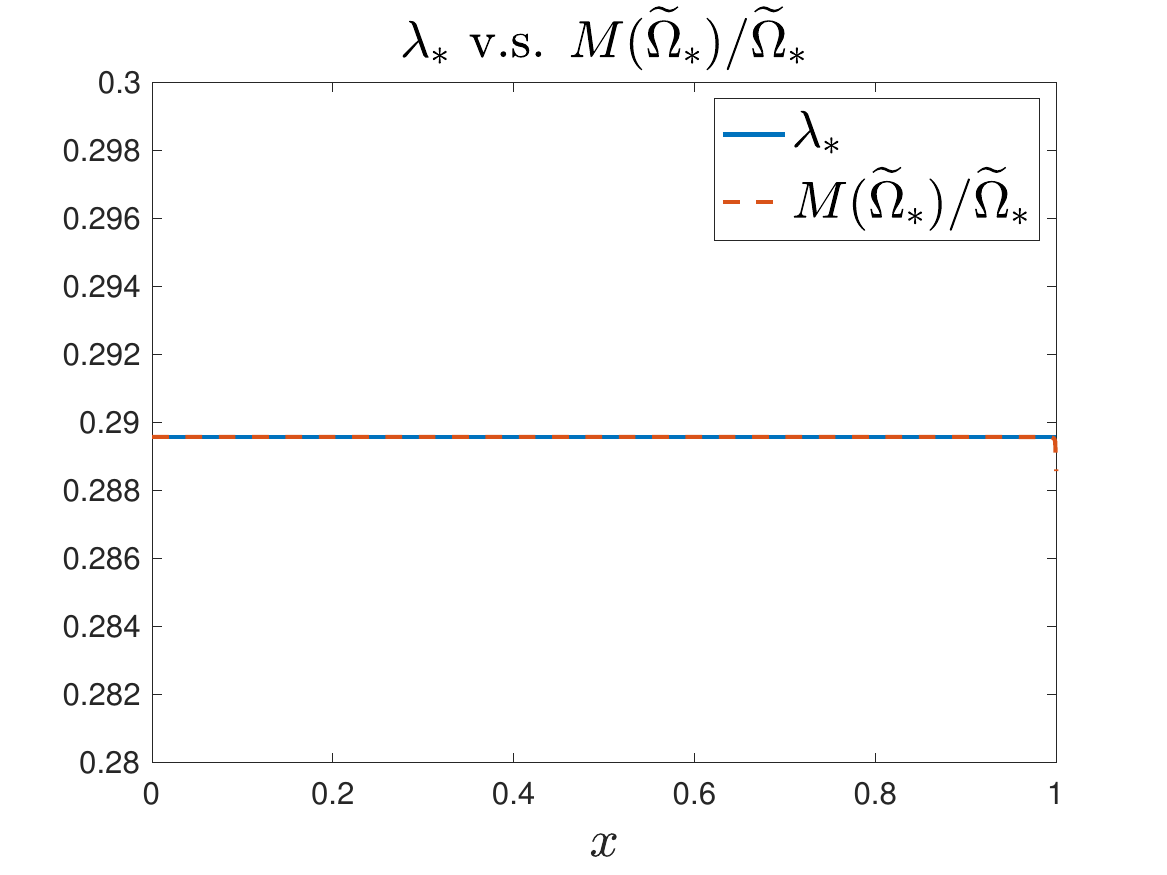}
        \caption{$\lambda_*$ v.s. $\mtx{M}(\widetilde{\Omega}_*)/\widetilde{\Omega}_*$}
    \end{subfigure}
    \caption[Profile Compare]{Comparison between the leading eigenfunction $f_*$ of $\mtx{M}$ and the rescaled approximate self-similar profile (denoted by $\widetilde{\Omega}_*$) numerically obtained in \cite{chen2021finite}. (a) The two profiles are hardly distinguishable. (b) The difference between the two profiles is small. (c) $\widetilde{\Omega}_*$ approximately solves the eigenvalue problem \eqref{eqt:eigen_problem} with the ratio $\mtx{M}(\widetilde{\Omega}_*)/\widetilde{\Omega}_*$ uniformly close to $\lambda_*=\lambda_1$ (c.f. Figure \ref{fig:Eigenfunctions}). Note that the numerical solution of $f_*$ is obtained by directly solving the eigenvalue problem \eqref{eqt:eigen_problem}, while the numerical solution $\widetilde{\Omega}_*$ is obtained in \cite{chen2021finite} by numerically solving the dynamic rescaling equation \eqref{eqt:dynamic_rescaling} of the De Gregorio model.}  
     \label{fig:profile_compare}
\end{figure}

\bibliographystyle{myalpha}
\bibliography{reference}

\end{document}